\documentclass[reqno,centertags, 12pt]{amsart}
\usepackage{amsmath,amsthm,amscd,amssymb}
\usepackage{latexsym}
\newcommand{\bbC}{{\mathbb{C}}}
\newcommand{\bbD}{{\mathbb{D}}}

\newcommand{\bbL}{{\mathbb{L}}}

\newcommand{\bbR}{{\mathbb{R}}}
\newcommand{\bbS}{{\mathbb{S}}}

\newcommand{\bbZ}{{\mathbb{Z}}}

\newcommand{\fre}{{\frak{e}}}

\newcommand{\dott}{\,\cdot\,}

\newcommand{\lb}{\label}
\newcommand{\f}{\frac}

\newcommand{\ol}{\overline}
\newcommand{\ti}{\tilde  }
\newcommand{\wti}{\widetilde  }

\newcommand{\Av}{\text{\rm{Av}}}
\newcommand{\Arg}{\text{\rm{Arg}}}
\newcommand{\dist}{\text{\rm{dist}}}

\newcommand{\ess}{\text{\rm{ess}}}
\newcommand{\ac}{\text{\rm{ac}}}

\newcommand{\s}{\text{\rm{s}}}

\newcommand{\supp}{\text{\rm{supp}}}

\newcommand{\bi}{\bibitem}

\newcommand{\beq}{\begin{equation}}
\newcommand{\eeq}{\end{equation}}
\newcommand{\ba}{\begin{align}}
\newcommand{\ea}{\end{align}}
\newcommand{\veps}{\varepsilon}

\newcommand{\comm}[1]{}
\def\R{{\mathbb{R}}}
\def\C{{\mathbb{C}}}
\def\K{{\mathcal {L}}}

\newcounter{smalllist}
\newenvironment{SL}{\begin{list}{{\rm\roman{smalllist})}}{%
\setlength{\topsep}{0mm}\setlength{\parsep}{0mm}\setlength{\itemsep}{0mm}%
\setlength{\labelwidth}{2em}\setlength{\leftmargin}{2em}\usecounter{smalllist}%
}}{\end{list}}

\DeclareMathOperator{\Real}{Re}
\DeclareMathOperator{\Ima}{Im}

\allowdisplaybreaks
\numberwithin{equation}{section}

\newtheorem{theorem}{Theorem}[section]
\newtheorem*{t1}{Theorem 1}
\newtheorem*{t2}{Theorem 2}
\newtheorem*{t3}{Theorem 3}
\newtheorem*{t4}{Theorem 4}
\newtheorem*{t5}{Theorem 5}
\newtheorem{proposition}[theorem]{Proposition}
\newtheorem{lemma}[theorem]{Lemma}
\newtheorem{corollary}[theorem]{Corollary}
\theoremstyle{definition}

\theoremstyle{remark}
\newtheorem*{remark}{Remark}
\newtheorem*{remarks}{Remarks}
\newtheorem*{definition}{Definition}
\newtheorem*{conjecture}{Conjecture}

\newcommand{\abs}[1]{\lvert#1\rvert}

\newcommand{\norm}[1]{\lVert#1\rVert}

\begin{document}
\title[Universality for Ergodic Jacobi Matrices]
{Bulk Universality and Clock Spacing\\
of Zeros for Ergodic Jacobi Matrices\\
with A.C.\ Spectrum}
\author[A.~Avila, Y.~Last, and B.~Simon]{Artur Avila$^1$, Yoram Last$^{2,4}$,
and Barry Simon$^{3,4}$}

\thanks{$^1$ CNRS UMR 7599, Laboratoire de Probabilit\'es et Mod\`eles Al\'eatoires,
Universit\'e Pierre et Marie Curie--Bo\^{i}te Courrier 188, 75252--Paris Cedex 05, France.
Current address: IMPA, Estrada Dona Castorina 110, Rio de Janeiro,
22460-320, Brazil.
E-mail: artur@math.sunysb.edu.}

\thanks{$^2$ Institute of Mathematics, The Hebrew University, 91904 Jerusalem, Israel.
E-mail: ylast@math.huji.ac.il. Supported in part by The Israel Science
Foundation (grant no.\ 1169/06)}

\thanks{$^3$ Mathematics 253-37, California Institute of Technology, Pasadena, CA 91125, USA.
E-mail: bsimon@caltech.edu. Supported in part by NSF grant DMS-0652919}

\thanks{$^4$ Research supported in part
by Grant No.\ 2006483 from the United States-Israel Binational Science
Foundation (BSF), Jerusalem, Israel}

\date{September 29, 2008}

\begin{abstract} By combining some ideas of Lubinsky with some soft analysis, we prove that universality and
clock  behavior of zeros for OPRL in the a.c.\ spectral region is implied by convergence of $\f{1}{n} K_n(x,x)$
for the diagonal CD kernel and boundedness of the analog associated to second kind polynomials. We then show
that these hypotheses are always valid for ergodic Jacobi matrices with a.c.\ spectrum and prove that the
limit of $\f{1}{n} K_n(x,x)$ is $\rho_\infty(x)/w(x)$ where $\rho_\infty$ is the density of zeros and $w$ is
the a.c.\ weight of the spectral measure.
\end{abstract}

\maketitle

\section{Introduction} \lb{s1}

Given a finite measure, $d\mu$, of compact and not finite support on $\bbR$, one defines the orthonormal
polynomials, $p_n(x)$ (or $p_n (x,d\mu)$ if the $\mu$-dependence is important), by applying Gram--Schmidt
to $1,x,x^2, \dots$. Thus, $p_n$ is a polynomial of degree exactly $n$ with leading positive coefficient so
that
\begin{equation} \lb{1.1}
\int p_n(x) p_m(x)\, d\mu(x) = \delta_{nm}
\end{equation}
See \cite{SzBk,FrBk,Rice} for background on these OPRL (orthogonal polynomials on the real line).

Associated to $\mu$ is a family of Jacobi parameters $\{a_n,b_n\}_{n=1}^\infty$, $a_n >0$, $b_n$ real,
determined by the recursion relation ($p_{-1}(x)\equiv 0$)
\begin{equation} \lb{1.2}
xp_n(x) = a_{n+1} p_{n+1}(x) + b_{n+1} p_n(x) + a_n p_{n-1}(x)
\end{equation}
The $\{p_n(x)\}_{n=0}^\infty$ are an orthonormal basis of $L^2(\bbR,d\mu)$ (since $\supp(d\mu)$ is
compact) and \eqref{1.2} says that multiplication by $x$ is given in this basis by the tridiagonal
Jacobi matrix
\begin{equation} \lb{1.3}
J=
\begin{pmatrix}
b_1 & a_1 & 0 &  \cdots \\
a_1 & b_2 & a_2 &  \cdots \\
0 & a_2 & b_3 &  \cdots \\
\vdots & \vdots & \vdots  & \ddots
\end{pmatrix}
\end{equation}

If we restrict (as we normally will) to $\mu$ normalized by $\mu(\bbR)=1$, then $\mu$ can be recovered
from $J$ as the spectral measure for the vector $(1,0,0,\dots)^t$. Favard's theorem says there is a
one-one correspondence between sets of bounded Jacobi parameters, that is,
\begin{equation} \lb{1.4}
\sup_n\, \abs{a_n} =\alpha_+ <\infty \qquad \sup_n \, \abs{b_n}=\beta <\infty
\end{equation}
and probability measures with compact and not finite support under this $\mu\to J\to\mu$ correspondence.

We will use this to justify spectral theory notation for things like $\supp(d\mu)$ which we will denote
$\sigma(d\mu)$ since it is the spectrum of $J$, $\sigma(J)$. We will use $\sigma_\ess(d\mu)$ for the
essential spectrum, and if
\begin{equation} \lb{1.5}
d\mu(x) =w(x)\, dx + d\mu_\s (x)
\end{equation}
where $d\mu_\s$ is Lebesgue singular, then we define
\begin{equation} \lb{1.6}
\Sigma_\ac (d\mu) =\{x\mid w(x) >0\}
\end{equation}
determined up to sets of Lebesgue measure $0$, so $\Sigma_\ac\neq\emptyset$ means $d\mu$ has a
nonvanishing a.c.\ part.

We will also suppose
\begin{equation} \lb{1.7}
\inf_n\, a_n = \alpha_- >0
\end{equation}
which is no loss since it is known \cite{Dom78} that if the $\inf$ is $0$, then $\Sigma_\ac =\emptyset$,
and we will only be interested in cases where $\Sigma_\ac \neq\emptyset$.

One of our concerns in this paper is the zeros of $p_n(x,d\mu)$. These are not only of intrinsic interest;
they enter in Gaussian quadrature and also as the eigenvalues of $J_{n;F}$, the upper left $n\times n$
corner of $J$, and so, relevant to statistics of eigenvalues in large boxes, a subject on which there is
an enormous amount of discussion in both the mathematics and the physics literature.

These zeros are all simple and lie in $\bbR$. $d\nu_n$ is the normalized counting measure for the zeros,
that is,
\begin{equation} \lb{1.8}
\nu_n(S)=\f{1}{n}\, \#(\text{zeros of $p_n$ in $S$})
\end{equation}
In many cases, $d\nu_n$ converges to a weak limit, $d\nu_\infty$, called the density of zeros or density
of states (DOS). If this weak limit exists, we say that the DOS exists. It often happens that $d\nu_\infty$
is $d\rho_\fre$, the equilibrium measure for $\fre =\sigma_\ess(d\mu)$. This is true, for example, if
$\rho_\fre$ is equivalent to $dx\restriction\fre$ and $\Sigma_\ac=\fre$, a theorem of Widom \cite{Wid}
and Van Assche \cite{vA} (see also Stahl--Totik \cite{StT} and Simon \cite{EqMC}). If $d\nu_\infty$ has
an a.c.\ part, we use $\rho_\infty (x)$ for $d\nu_\infty /dx$ and we use $\rho_\fre(x)$ for $d\rho_\fre/dx$.
More properly, $d\nu_\infty$ is the ``density of states measure'' (so $\int_{-\infty}^x d\nu_\infty$ is
the ``integrated density of states'') and $\rho_\infty (x)$ the ``density of
states.''

We are especially interested in the fine structure of the zeros near some point $x_0\in\sigma(d\mu)$. We
define $x_j^{(n)}(x_0)$ by
\begin{equation} \lb{1.9}
x_{-2}^{(n)}(x_0) < x_{-1}^{(n)}(x_0) < x_0 \leq x_0^{(n)}(x_0) < x_1^{(n)}(x_0) < \dots
\end{equation}
requiring these to be all of the zeros near $x_0$. It is known that if $x_0$ is not isolated from
$\sigma(d\mu)$ on either side, that is, for all $\delta >0$,
\begin{equation} \lb{1.10}
(x_0-\delta, x_0) \cap \sigma(d\mu) \neq \emptyset \neq (x_0, x_0+\delta) \cap \sigma(d\mu)
\end{equation}
then for each fixed $j$,
\begin{equation} \lb{1.11}
\lim_{n\to\infty}\, x_j^{(n)}(x_0) =x_0
\end{equation}
We are interested in clock behavior named after the spacing of numerals on a clock---meaning equal spacing
of the zeros nearby to $x_0$:

\begin{definition} We say that there is {\it quasi-clock behavior\/} at $x_0\in\sigma(d\mu)$ if and only
if for each fixed $j\in\bbZ$,
\begin{equation} \lb{1.12}
\lim_{n\to\infty}\, \f{x_{j+1}^{(n)}(x_0) - x_j^{(n)}(x_0)}{x_1^{(n)}(x_0) -x_0^{(n)}(x_0)} =1
\end{equation}
We say there is {\it strong clock behavior\/} at $x_0$ if and only if the DOS exists and for each fixed $j\in\bbZ$,
\begin{equation} \lb{1.13}
\lim_{n\to\infty}\, n(x_{j+1}^{(n)}(x_0) - x_j(x_0)) = \f{1}{\rho_\infty(x_0)}
\end{equation}
\end{definition}

Obviously, strong clock behavior implies quasi-clock behavior. Thus far, the only cases where it is proven there is
quasi-clock behavior, one has strong clock behavior but, as we will explain in Section~\ref{s7}, we think there are
examples where one has quasi-clock behavior at $x_0$ but not strong clock behavior. Before this paper, all examples
known with strong clock behavior have $\rho_\infty =\rho_\fre$, but we will find several examples where there is
strong clock behavior with $\rho_\infty \neq \rho_\fre$ in Section~\ref{s7}. In that section, we will say more about:

\begin{conjecture} For any $\mu$, quasi-clock behavior holds at a.e.\ $x_0\in\Sigma_\ac(d\mu)$.
\end{conjecture}

In this paper, one of our main goals is to prove this result for ergodic Jacobi matrices. A major role will
be played by the CD (for Christoffel--Darboux) kernel, defined for $x,y\in\bbC$ by
\begin{equation} \lb{1.14}
K_n(x,y) =\sum_{j=0}^n \, \ol{p_j(x)}\, p_j(y)
\end{equation}
the integral kernel for the orthogonal projection onto polynomials of degree at most $n$ in $L^2 (\bbR,d\mu)$;
see Simon \cite{CD} for  a review of some important aspects of the properties and uses of this kernel.
We will repeatedly make use of the CD formula,
\begin{equation} \lb{1.15}
K_n(x,y) = \f{a_{n+1} [\, \ol{p_{n+1}(x)}\, p_n(y) - \ol{p_n(x)}\, p_{n+1}(y)]}{\bar
x-y},
\end{equation}
the Schwarz inequality,
\begin{equation} \lb{1.16}
\abs{K_n(x,y)}^2 \leq K_n(x,x) K_n(y,y)
\end{equation}
and the reproducing property,
\begin{equation} \lb{1.17}
\int K_n(x,y) K_n(y,z)\, d\mu(y) =K_n(x,z).
\end{equation}

It is a theorem (see Simon \cite{weak-CD}) that if the DOS exists, then
\begin{equation} \lb{1.18}
\f{1}{n+1}\, K_n(x,x)\, d\mu(x) \overset{\text {weak}}{\longrightarrow} d\nu_\infty(x)
\end{equation}
and, in general, $\f{1}{n+1} K_n(x,x)\, d\mu(x)$ has the same weak limit points as $d\nu_n$. This suggests
that a.c.\ parts converge pointwise, that is, one hopes that for a.e.\ $x_0\in\Sigma_\ac$,  
\begin{equation} \lb{1.19}
\f{1}{n+1}\, K_n(x_0,x_0) \to \f{\rho_\infty(x_0)}{w(x_0)}
\end{equation}
This has been proven for regular (in the sense of Stahl--Totik \cite{StT}; see also Simon \cite{EqMC})
measures with a local Szeg\H{o} condition in a series of papers of which the seminal ones are
M\'at\'e--Nevai-Totik \cite{MNT91} and Totik \cite{Tot}. We will prove it for ergodic Jacobi matrices.

We say {\it bulk universality\/} holds at $x_0\in\supp(d\mu)$ if and only if uniformly for $a,b$ in
compact subsets of $\bbR$, we have
\begin{equation} \lb{1.20}
\f{K_n(x_0 + \f{a}{n}, x_0 + \f{b}{n})}{K_n(x_0,x_0)} \to \f{\sin(\pi\rho(x_0)(b-a))}{\pi\rho(x_0)(b-a)}
\end{equation}
We use the term ``bulk'' here because \eqref{1.20} fails at edges of the spectrum; see Lubinsky
\cite{Lub2008}. We also note that when \eqref{1.20} holds, typically (and in all cases below) for $z,w$
complex, one has
\begin{equation} \lb{1.21}
\f{K_n(x_0 + \f{z}{n}, x_0 + \f{w}{n})}{K_n(x_0,x_0)} \to
\f{\sin(\rho(x_0)(w-\bar z))}{\rho(x_0) (w-\bar z)}
\end{equation}

Freud \cite{FrBk} proved bulk universality for measures on $[-1,1]$ with $d\mu_\s =0$ and strong
conditions on $w(x)$. Because of related results (but with variable weights) in random matrix theory,
this result was re-examined and proven in multiple interval support cases with analytic weights by
Kuijlaars--Vanlessen \cite{KV}. A significant breakthrough was made by Lubinsky \cite{Lub}, whose
contributions we return to shortly.

It is a basic result of Freud \cite{FrBk}, rediscovered by Levin (in \cite{LL}), that

\begin{theorem}[Freud--Levin Theorem]\lb{T1.1} Bulk universality at $x_0$ implies strong clock behavior at $x_0$.
\end{theorem}

\begin{remarks} 1. The proof (see \cite{FrBk,LL,CD}) relies on the CD formula, \eqref{1.15}, which
implies that if $y_0$ is a zero of $p_n$, then the other zeros of $p_n$ are the points $y$ solving
$K_n(y,y_0)=0$ and the fact that the zeros of $\sin (\pi\rho(x_0)(b-a))$ are at $b-a = j/\rho(x_0)$
with $j\in\bbZ$.

\smallskip
2. Szeg\H{o} \cite{SzBk} proved strong clock behavior for Jacobi polynomials and Erd\H{o}s--Tur\'an \cite{ET}
for a more general class of measures on $[-1,1]$. Simon \cite{Fine1,Fine2,Fine3,Fine4} has a series on
the subject. The paper with Last \cite{Fine4} was one motivation for Levin--Lubinsky \cite{LL}.

\smallskip
3. Lubinsky (private communication) has emphasized to us that this part of \cite{LL} is due to Levin
alone---hence our name for the result.
\end{remarks}

It is also useful to define
\begin{equation} \lb{1.22}
\rho_n = \f{1}{n}\, w(x_0) K_n(x_0,x_0)
\end{equation}
so \eqref{1.19} is equivalent to
\begin{equation} \lb{1.23}
\rho_n \to \rho_\infty (x_0)
\end{equation}
We say {\it weak bulk universality} holds at $x_0$ if and only if, uniformly for $a,b$ on compact subsets of
$\bbR$, we have
\begin{equation} \lb{1.24}
\f{K_n(x_0 + \f{a}{n\rho_n}, x_0 + \f{b}{n\rho_n})}{K_n(x_0,x_0)} \to
\f{\sin(\pi (b-a))}{\pi(b-a)}
\end{equation}
the form in which universality is often written, especially in the random matrix literature. Notice that
\begin{equation} \lb{1.25}
\text{weak universality} + \text{\eqref{1.23}} \Rightarrow \text{universality}
\end{equation}
Notice also that \eqref{1.24} could hold in case where $\rho_n$ does not converge as
$n\to\infty$. The same proof that verifies Theorem~\ref{T1.1} implies

\begin{theorem}[Weak Freud--Levin Theorem] \lb{T1.2} Weak bulk universality at $x_0$ implies quasi-clock
behavior at $x_0$.
\end{theorem}

With this background in place, we can turn to describing the main results of this paper: five theorems, proven
one per section in Sections~\ref{s2}--\ref{s6}.

The first theorem is an abstraction, extension, and simplification of Lubinsky's second approach to universality
\cite{Lub-jdam}. In \cite{Lub}, Lubinsky found a beautiful way of going from control of the diagonal CD kernel
to the off-diagonal (i.e., to universality). It depended on the ability to control limits not only of $\f{1}{n}
K_n(x_0,x_0)$ but also $\f{1}{n} K_n(x_0 + \f{a}{n}, x_0 +\f{a}{n})$---what we call the Lubinsky wiggle.
We will especially care about the {\it Lubinsky wiggle condition}:
\begin{equation} \lb{1.26}
\lim_{n\to\infty}\, \f{K_n(x_0 + \f{a}{n}, x_0 + \f{a}{n})}{K_n(x_0,x_0)} =1
\end{equation}
uniformly for $a\in [-A,A]$ for each $A$. In addition to this, in \cite{Lub}, Lubinsky needed a simple but
clever inequality and, most significantly, a comparison model example where one knows universality holds.
For $[-1,1]$, he took Legendre polynomials (i.e., $d\mu = \f12 \chi_{[-1,1]}(x)\,dx$). In extending this to
more general sets, one uses approximation by finite gap sets as pioneered by Totik \cite{Tot-acta}. Simon
\cite{2exts} then used Jacobi matrices in isospectral tori for a comparison model on these finite gap sets,
while Totik \cite{Tot-prep} used polynomials mappings and the results for $[-1,1]$.

For ergodic Jacobi matrices where $\sigma(d\mu)$ is often a Cantor set, it is hard to find comparison
models, so we will rely on a second approach developed by Lubinsky \cite{Lub-jdam} that seems to be able to
handle any situation that his first approach can and which does not rely on a comparison model. Our first
theorem, proven in Section~\ref{s2}, is a variant of this approach. We need a preliminary definition:

\begin{definition} Let $d\mu$ be given by \eqref{1.5}. A point $x_0$ is called a {\it Lebesgue point} of
$d\mu$ if and only if $w(x_0) >0$ and
\begin{align}
\lim_{\delta\downarrow 0}\, (2\delta)^{-1} \int_{x_0-\delta}^{x_0+\delta} \abs{w(x)-w(x_0)}\, dx &=0 \lb{1.27} \\
\lim_{\delta\downarrow 0}\, (2\delta)^{-1} \mu_\s (x_0 -\delta, x_0 + \delta) &= 0 \lb{1.28}
\end{align}
\end{definition}

Standard maximal function methods (see Rudin \cite{Rudin}) show Lebesgue a.e.\ $x_0\in\Sigma_\ac (d\mu)$ is a
Lebesgue point.

\begin{t1} Let $x_0$ be a Lebesgue point of $\mu$. Suppose that
\begin{SL}
\item[{\rm{(1)}}] The Lubinsky wiggle condition \eqref{1.26} holds uniformly for $a\in [-A,A]$ and any
$A<\infty$.

\item[{\rm{(2)}}] We have
\begin{equation} \lb{1.29}
\liminf_{n\to\infty}\, \f{1}{n+1}\, K_n (x_0,x_0) >0
\end{equation}

\item[{\rm{(3)}}] For any $\veps$, there is $C_\veps >0$ so that for any $R<\infty$, there is an $N$ so
that for all $n>N$ and all $z\in\bbC$ with $\abs{z}<R$, we have
\begin{equation} \lb{1.30}
\f{1}{n+1}\, K_n \biggl(x_0 + \f{z}{n}\, , x_0+ \f{z}{n}\biggr) \leq C_\veps \exp (\veps\abs{z}^2)
\end{equation}
\end{SL}
Then weak bulk universality, and so, quasi-clock behavior, holds at $x_0$.
\end{t1}

\begin{remarks} 1. If one replaces \eqref{1.30} by
\begin{equation} \lb{1.31}
C \exp (A\abs{z})
\end{equation}
then the result can be proven by following Lubinsky's argument in \cite{Lub-jdam}. He does not assume \eqref{1.31}
directly but rather hypotheses that he shows imply it (but which is invalid in case $\supp(d\mu)$ is a Cantor
set).

\smallskip
2. Because our Theorem~3 below is so general, we doubt there are examples where \eqref{1.30} holds but \eqref{1.31}
does not, but we feel our more general abstract result is clarifying.

\smallskip
3. The strategy we follow is Lubinsky's, but the tactics differ and, we feel, are more elementary and illuminating.
\end{remarks}

In \cite{Lub-jdam}, the only examples where Lubinsky can verify his wiggle condition are the situations where Totik
\cite{Tot-prep} proves universality using Lubinsky's first method. To go beyond that, we need the following, proven
in Section~\ref{s3}:

\begin{t2} Let $\Sigma\subset\Sigma_\ac$. Suppose for a.e.\ $x_0\in\Sigma$, we have that condition~{\rm{(3)}}
of Theorem~1 holds and that
\begin{SL}
\item[{\rm{(4)}}] $\lim_{n\to\infty} \f{1}{n+1} K_n(x_0,x_0)$ exists and is strictly positive.
\end{SL}
Then condition {\rm{(1)}} of Theorem~1 holds for a.e.\ $x_0\in\Sigma$.
\end{t2}

Of course, (4) implies condition (2). So we obtain:

\begin{corollary}\lb{L1.3} If {\rm{(3)}} and {\rm{(4)}} hold for a.e.\ $x_0\in\Sigma$, then for a.e. $x_0\in\Sigma$,
we have weak universality and quasi-clock behavior.
\end{corollary}

By \eqref{1.25}, we see

\begin{corollary}\lb{C1.4} If {\rm{(3)}} and {\rm{(4)}} hold for a.e.\ $x_0\in\Sigma$, and if the DOS exists
and the limit in {\rm{(4)}} is $\rho_\infty(x)/w(x)$, then for a.e.\ $x\in \Sigma$, we have universality and
strong clock behavior.
\end{corollary}

Next, we need to examine when \eqref{1.30} holds. We will not only obtain a bound of the type \eqref{1.31} but
one that does not need to vary $N$ with $R$ and is universal in $z$. We will use transfer matrix techniques
and notation.

Given Jacobi parameters, $\{a_n,b_n\}_{n=1}^\infty$, we define
\begin{equation} \lb{1.32}
A_j(z) = \begin{pmatrix}
\f{z-b_j}{a_j} & - \f{1}{a_j} \\
a_j & 0
\end{pmatrix}
\end{equation}
so that \eqref{1.2} is equivalent to
\begin{equation} \lb{1.33}
\begin{pmatrix}
p_n(x) \\ a_n p_{n-1} (x)
\end{pmatrix}
= A_n(x)
\begin{pmatrix}
p_{n-1}(x) \\ a_{n-1} p_{n-2}(x)
\end{pmatrix}
\end{equation}
We normalize, placing $a_n$ on the lower component, so that
\begin{equation} \lb{1.34}
\det(A_j(z)) =1
\end{equation}

The transfer matrix is then defined by
\begin{equation} \lb{1.35}
T_n(z) = A_n(z) \dots A_1(z)
\end{equation}
so
\begin{equation} \lb{1.36}
\begin{pmatrix} p_n(x) \\ a_n p_{n-1}(x)
\end{pmatrix}
=T_n(x)
\begin{pmatrix} 1 \\ 0
\end{pmatrix}
\end{equation}
If $\ti p_n$ are the OPRL associated to the once stripped Jacobi parameters $\{a_{n+1}, b_{n+1}\}_{n=1}^\infty$,
and
\begin{equation} \lb{1.37}
q_n(x) = -a_1^{-1} \ti p_{n-1}(x)
\end{equation}
with $q_0=0$, then
\begin{equation} \lb{1.38}
T_n(z) =
\begin{pmatrix}
p_n(z) & q_n(z) \\
a_n p_{n-1}(z) & a_n q_{n-1}(z)
\end{pmatrix}
\end{equation}

Here is how we will establish \eqref{1.30}/\eqref{1.31}:

\begin{t3} Fix $x_0\in\bbR$. Suppose that
\begin{equation} \lb{1.39}
\sup_n  \f{1}{n+1} \sum_{j=0}^n\, \norm{T_j(x_0)}^2 \leq C <\infty
\end{equation}
Then for all $z\in\bbC$ and all $n$,
\begin{equation} \lb{1.40}
\f{1}{n+1} \sum_{j=0}^n\, \biggl\| T_j\biggl(x_0 + \f{z}{n+1}\biggr)\biggr\|^2 \leq
C \exp (2C\alpha_-^{-1} \abs{z})
\end{equation}

Moreover, if
\begin{equation} \lb{1.41}
\sup_n\, \norm{T_n(x_0)}^2 = C <\infty
\end{equation}
then for all $z\in\bbC$ and $n$,
\begin{equation} \lb{1.42}
\biggl\| T_n \biggl( x_0 + \f{z}{n+1}\biggr)\biggr\| \leq C^{1/2} \exp (C\alpha_-^{-1} \abs{z})
\end{equation}
\end{t3}

\begin{remarks} 1. Our proof is an abstraction of ideas of Avila--Krikorian \cite{AvKr} who
only treated the ergodic case.

\smallskip
2. $\alpha_-$ is given by \eqref{1.7}.

\smallskip
3. There is a conjecture, called the Schr\"odinger conjecture (see \cite{MMG}), that says \eqref{1.41}
holds for a.e.\ $x_0\in\Sigma_\ac(d\mu)$.
\end{remarks}

Our last two theorems below are special to the ergodic situation. Let $\Omega$ be a compact metric space, $d\eta$
a probability measure on $\Omega$, and $S\colon\Omega\to\Omega$ an ergodic invertible map of $\Omega$ to
itself. Let $A,B$ be continuous real-valued functions on $\Omega$ with $\inf_\omega A(\omega) >0$. Let
\begin{equation} \lb{1.43}
\alpha_+ = \norm{A}_\infty \qquad
\beta =\norm{B}_\infty \qquad
\alpha_- = \norm{A^{-1}}_\infty^{-1}
\end{equation}
For each $\omega\in\Omega$, $J_\omega$ is the Jacobi matrix with
\begin{equation} \lb{1.44}
a_n(\omega) = A(S^{n-1}\omega) \qquad
b_n(\omega) = B(S^{n-1}\omega)
\end{equation}
\eqref{1.43} is consistent with \eqref{1.4} and \eqref{1.7}. Usually one only takes $\Omega$, a measure
space, and $A,B$ bounded measurable functions, but by replacing $\Omega$ by $([\alpha_-,\alpha_+]\times
[-\beta,\beta])^\infty\equiv \wti\Omega$ and mapping $\Omega\to\wti\Omega$ by $\omega\mapsto (A(S^n\omega),
B(S^n\omega))_{n=-\infty}^\infty$, we get a compact space model equivalent to the original measure model.
We use $d\mu_\omega$ for the spectral measure of $J_\omega$ and $p_n(x,\omega)$ for $p_n(x,d\mu_\omega)$.

The canonical example of the setup with a.c.\ spectrum is the almost Mathieu equation. $\alpha$ is a fixed
irrational, $\lambda$ a nonzero real, $\Omega=\partial\bbD$, the unit circle $\{e^{i\theta}\mid\theta\in
[0,2\pi)\}$
\[
a_n\equiv 1 \qquad b_n=2\lambda\cos(\pi\alpha n + \theta)
\]
(so $S(e^{i\theta})=e^{i\theta} e^{i\pi\alpha}$, $d\eta(\theta) = d\theta/2\pi$). If $0\neq\abs{\lambda} <1$,
it is known (see \cite{Av-prep1,AD,AJ,Jit07}) that the spectrum is purely a.c.\ and is a Cantor set. 
It is also known \cite{Jit07} that if $\abs{\lambda} \geq 1$, there is no a.c.\ spectrum.

We will prove the following in Section~\ref{s5}:

\begin{t4} Let $\{J_\omega\}_{\omega\in n}$ be an ergodic family with $\Sigma_\ac$, the common
essential support of the a.c.\ spectrum
of $J_\omega$, of positive Lebesgue measure. Then for a.e.\ pairs $(x,\omega)\in\Sigma_\ac\times\Omega$,
\begin{alignat}{2}
& \text{\rm{(i)}} \qquad && \lim_{n\to\infty} \f{1}{n+1} \sum_{j=0}^n \, \abs{p_j(x,w)}^2
 \text{ exists} \lb{1.45} \\
& \text{\rm{(ii)}} && \lim_{n\to\infty} \f{1}{n+1} \sum_{j=0}^n \, \abs{q_j(x,w)}^2 \text{ exists} \notag
\end{alignat}
\end{t4}

In Section~\ref{s6}, we will prove

\begin{t5} For a.e.\ $(x,\omega)$ in $\Sigma_\ac\times\Omega$, the limit in  \eqref{1.45} is
$\rho_\infty(x)/w_\omega(x)$ where $\rho_\infty$ is the density of the a.c.\ part of the DOS.
\end{t5}

\noindent {\it Note}. This is, of course, an analog of the celebrated results of M\'at\'e--Nevai--Totik
\cite{MNT91} (for $[-1,1]$) and Totik \cite{Tot} (for general sets $\fre$ containing open intervals) for
regular measures obeying a local Szeg\H{o} condition.

\smallskip

Theorems~3--5 show the applicability of Theorem~2, and so lead to

\begin{corollary}\lb{C1.5} For any ergodic Jacobi matrix, we have universality and strong clock behavior for
a.e.\ $\omega$ and a.e.\ $x_0\in\Sigma_\ac$.
\end{corollary}

In particular, the almost Mathieu equation has strong clock behavior for the zeros.

\begin{remark}

It is possible to show that for the almost Mathieu equation
there is universality for a.e. $x_0 \in \Sigma_\ac$ and {\it every}
$\omega$.  Our current approach to this uses that the Schr\"odinger
conjecture is true for the almost Mathieu operator, a recently announced
result \cite {AFK}.

\end{remark}

For $n=1,2,3,4,5$, Theorem~$n$ is proven in Section~$n+1$. Section~\ref{s7} has some further remarks.

\medskip
\noindent{\bf Acknowledgments.} A.A.\ would like to thank M.~Flach and T.~Tombrello for the hospitality
of Caltech. B.S.\ would like to thank E.~de~Shalit for the hospitality of Hebrew Universality. This
research was partially conducted during the period A.A.\ served as a Clay Research Fellow. We would
like to thank H.~Furstenberg and B.~Weiss for useful comments.

\section{Lubinsky's Second Approach} \lb{s2}

In this section, we will prove Theorem~1. We begin with two overall visions relevant to the proof. First,
the so-called ``sinc kernel'' \cite{LB}, $\sin\pi z/\pi z$ enters as the Fourier transform of a suitable
multiple of the characteristic function of $[-\pi,\pi]$.

Second, the ultimate goal of quasi-clock spacing is that on a $1/n\rho_n$ scale, zeros are a unit distance
apart, so on this scale
\begin{equation} \lb{2.1}
\# \text{ of zeros in } [0,n] \sim n
\end{equation}
Lubinsky's realization is that the Lubinsky wiggle condition and Markov--Stieltjes inequalities (see below)
imply the difference of the two sides of \eqref{2.1} is bounded by $1$. This is close enough that,
together with some complex variable magic, one gets unit spacing.

The complex variable magic is encapsulated in the following result whose proof we defer until the end of
the section:

\begin{theorem}\lb{T2.1} Let $f$ be an entire function with the following properties:
\begin{SL}
\item[{\rm{(a)}}]
\begin{equation} \lb{2.2}
 f(0)=1
\end{equation}

\item[{\rm{(b)}}]
\begin{equation} \lb{2.3}
\sup_{x\in\bbR} \, \abs{f(x)}<\infty
\end{equation}

\item[{\rm{(c)}}]
\begin{equation} \lb{2.4}
 \int_{-\infty}^\infty \abs{f(x)}^2 \, dx \leq 1
\end{equation}

\item[{\rm{(d)}}] $f$ is real on $\bbR$.

\item[{\rm{(e)}}] All the zeros of $f$ lie on $\bbR$ and if these zeros are labelled by
\begin{equation} \lb{2.5}
\ldots \leq z_{-2} \leq z_{-1} < 0 < z_1 \leq z_2 \leq \dots
\end{equation}
with $z_0\equiv 0$, then
\begin{equation} \lb{2.6}
\abs{z_j-z_k} \geq \abs{j-k} -1
\end{equation}

\item[{\rm{(f)}}] For each $\veps >0$, there is $C_\veps$ with
\begin{equation} \lb{2.7}
\abs{f(z)} \leq C_\veps e^{\veps\abs{z}^2}
\end{equation}
\end{SL}
Then
\begin{equation} \lb{2.8}
f(z) = \f{\sin (\pi z)}{\pi z}
\end{equation}
\end{theorem}

\begin{remarks} 1. \eqref{2.6} allows $f$ a priori to have double zeros but not triple or higher zeros.

\smallskip
2. It is easy to see there are examples where \eqref{2.7} holds for some but not all of $\veps$ and
where \eqref{2.8} is false, so \eqref{2.7} is sharp.
\end{remarks}

\begin{proof}[Proof of Theorem~2 given Theorem~\ref{T2.1}] (This part of the argument is essentially in
Lubinsky \cite{Lub-jdam}.) Fix $a\in\bbR$ and let
\begin{equation} \lb{2.9}
f_n(z) = \f{K_n(x_0 + \f{a}{n\rho_n}, x_0 + \f{a+z}{n\rho_n})}{K_n(x_0, x_0)}
\end{equation}
By \eqref{1.29}, \eqref{1.30}, and \eqref{1.16}, the $f_n$ are uniformly bounded on each disk $\{z\mid\abs{z}<R\}$, so
by Montel's theorem, we have compactness that shows it suffices to prove that any limit point $f(z)$ has
the form \eqref{2.8}. We will show that this putative limit point obeys conditions (a)--(f) of
Theorem~\ref{T2.1}.

By the Lubinsky wiggle condition \eqref{1.26}, (a) holds. By Schwarz inequality, \eqref{1.11} and the
wiggle condition,
\begin{equation} \lb{2.10}
\sup_{x\in\bbR}\, \abs{f(x)} =1
\end{equation}
which is stronger than (b).

By \eqref{1.17},
\begin{equation} \lb{2.12}
\int_{\abs{y-x_0-\f{a}{n\rho_n}} \leq \f{R}{n\rho_n}} \abs{K_n(x,y)}^2 w(y)\, dy \leq K_n(x,x)
\end{equation}
for each $R<\infty$. Changing variables and using the Lebesgue point condition leads to
\begin{equation} \lb{2.13}
\int_{-R}^R \abs{f(y)}^2\, dy \leq 1
\end{equation}
which yields \eqref{2.4} (see Lubinsky \cite{Lub-jdam} for
more details).
In this, one uses \eqref{1.29}
and \eqref{1.30} to see that
\begin{equation} \lb{2.14x}
0 < \inf \rho_n < \sup \rho_n <\infty.
\end{equation}

That $f$ is real on $\bbR$ is immediate; the reality of zeros follows from Hurwitz's theorem and the fact
(see, e.g., \cite{CD}) that $p_{n+1}(x) -cp_n(x)$ has only real zeros for $c$ real.

The Markov--Stieltjes inequalities (see \cite{Markov,FrBk,CD}) assert that if $x_1,x_2, \dots$
are successive zeros of $p_n(x)-cp_{n-1}(x)$ for some $c$, then for $j\geq k+2$,
\begin{equation} \lb{2.14}
\mu([x_j,x_k]) \geq \sum_{\ell=k+1}^{j-1} \f{1}{K_n(x_\ell,x_\ell)}
\end{equation}
Using the fact that the $z_j$ (including $z_0$) are, by Hurwitz's theorem, limits of $x_j$'s scaled
by $n\rho_n$ and the Lubinsky wiggle condition to control limits of $n\rho_n/K_n(x_\ell,x_\ell)$, one
finds (see Lubinsky \cite{Lub-jdam} for more details)
that \eqref{2.6} holds. Here one uses that $x_0$
is a Lebesgue point to be sure that
\begin{equation} \lb{2.15}
\f{1}{x_k -x_j} \int_{x_j}^{x_k} d\mu(y)\to w(x_0)
\end{equation}

Finally, \eqref{1.30} implies \eqref{2.7}. Thus, \eqref{2.8} holds.
\end{proof}

The following will reduce the proof of Theorem~\ref{T2.1} to using conditions (a)--(e) to improving
the bound \eqref{2.7}.

\begin{proposition}\lb{P2.2} \begin{SL}
\item[{\rm{(a)}}] Fix $a>0$. If $f$ is measurable, real-valued and
supported on $[-a,a]$ with
\begin{equation} \lb{2.16}
\int_{-a}^a f(x)^2\, dx \leq 2a \quad\text{and}\quad \int_{-a}^a f(x)\, dx = 2a
\end{equation}
then
\begin{equation} \lb{2.17}
f(x) =\chi_{[-a,a]} (x) \quad\hbox{a.e.}
\end{equation}

\item[{\rm{(b)}}] If $f$ is real-valued and continuous on $\bbR$ and $\widehat f$ is supported on
$[-\pi,\pi]$ with
\begin{equation} \lb{2.18}
\int_{-\infty}^\infty f(x)^2\, dx \leq 1 \quad\text{and}\quad f(0)=1
\end{equation}
then
\begin{equation} \lb{2.19}
f(x) =\f{\sin(\pi x)}{\pi x}
\end{equation}

\item[{\rm{(c)}}] If $f$ is an entire function, real on $\bbR$ with \eqref{2.18}, and for all $\delta >0$,
there is $C_\delta$ with
\begin{equation} \lb{2.20}
\abs{f(z)}\leq C_\delta \exp ((\pi+\delta)\abs{\Ima z})
\end{equation}
then \eqref{2.8} holds.
\end{SL}
\end{proposition}

\begin{proof} (a) \ Essentially this follows from equality in the Schwarz inequality. More precisely,
\eqref{2.16} implies
\begin{equation} \lb{2.21}
\int_{-a}^a \abs{f(x)-\chi_{[-a,a]}(x)}^2\, dx \leq 0
\end{equation}

\smallskip
(b) \ Apply part (a) to $(2\pi)^{1/2} \widehat f(k)$ with $a=\pi$.

\smallskip
(c) \ By the Paley--Wiener theorem, \eqref{2.20} implies that $\widehat f$ is supported on $[-\pi,\pi]$.
\end{proof}

Thus, we are reduced to going from \eqref{2.7} to \eqref{2.20}.

\comm{
The key tool is the Phragm\'en--Lindel\"of
principle in the following form (see Titchmarsh \cite[Sect.~5.61]{Tit}): Let $S$ be a sector of open angle
$\pi/\alpha$, that is, for some $z_0\in\partial\bbD$,
\begin{equation} \lb{2.22}
S=\{z\mid z=z_0 re^{i\theta};\, 0 < r <\infty,\, \abs{\theta} < \pi/2\alpha\}
\end{equation}
Suppose $f$ is analytic in $S$ and continuous on $\bar S$ and obeys
\begin{equation} \lb{2.23}
\abs{f(z)} \leq C\exp (\abs{z}^\beta)
\end{equation}
for some
\begin{equation} \lb{2.24}
\beta <\alpha
\end{equation}
Then
\begin{equation} \lb{2.25x}
\sup_{z\in\bar S}\, \abs{f(z)} \leq \sup_{z\in\partial S}\, \abs{f(z)}
\end{equation}
and, in particular, if $f$ is bounded on $\partial S$, it is bounded on $S$.

As a warmup, we prove a variant of Theorem~\ref{T2.1} that suffices for most applications, which is somewhat
simpler to prove and which we will need in the proof of Theorem~\ref{T2.1}:

\begin{theorem}\lb{T2.3} Suppose {\rm{(a)}}--{\rm{(d)}} of Theorem~\ref{T2.1} hold but {\rm{(e)}} is replaced
by the weaker and {\rm{(f)}} by the stronger
\begin{SL}
\item[{\rm{(e$^\prime$)}}] $\abs{z_j} \geq \abs{j}-1$
\item[{\rm{(f$^\prime$)}}] For some $\delta <2$, there is $C$ with
\begin{equation} \lb{2.25}
\abs{f(z)} \leq Ce^{\abs{z}^\delta}
\end{equation}
\end{SL}
Then \eqref{2.8} holds.
\end{theorem}

\begin{proof}
}

By $f(0)=1$, the reality of the zeros and \eqref{2.7}, we have, by the
Hadamard factorization theorem (see Titchmarsh
\cite[Sect.~8.24]{Tit}) that
\begin{equation} \lb{2.26}
f(z) = e^{Az} \prod_{j\neq 0} \biggl( 1-\f{z}{z_j}\biggr) e^{z/z_j}
\end{equation}
with $A$ real. For $x\in\bbR$, define $z_j(x)$ to be a renumbering of the $z_j$, so
\begin{equation} \lb{2.24x}
\ldots \leq z_{-1}(x) < x \leq z_0(x) \leq z_1(x) \leq \ldots
\end{equation}
By $\abs{z_j-z_k}\geq \abs{k-j}-1$, we see that
\begin{equation} \lb{2.25y}
z_{n+1}(x) - x \geq n \qquad x-z_{-(n+1)}(x) \geq n
\end{equation}

In particular, $(x-1.1,\, x+ 1.1)$ can contain at most $z_0(x), z_{\pm 1}(x),
z_{\pm 2}(x)$. Removing the open intervals of size $2/10$ about each of the five points $\abs{z_\ell(x)-x}$
($\ell=0,\pm 1, \pm 2$) from $[0,1]$ leaves at least one $\delta >0$, that is, we can pick $\delta(x)$
in $[0,1]$ so for all $j$,
\begin{equation} \lb{2.25a}
\abs{z_j(n)  - (x\pm \delta)} \geq \tfrac{1}{10}
\end{equation}
Moreover, by \eqref{2.25y}, for $n=1,2,\dots$,
\begin{equation} \lb{2.25b}
\abs{z_{\pm(n+2)}(x) - (x\pm\delta)} \geq n
\end{equation}

Since
\begin{equation} \lb{2.25c}
\f{\abs{1-(x+iy)/z_j}^2}{\abs{(1-(x+\delta/z_j)(1-x-\delta)/z_j)}} \leq 1 +
\f{(y^2 + \delta^2)}{\abs{z_j-(x+\delta)} \abs{z_j-(x-\delta)}}
\end{equation}
we conclude from \eqref{2.26} that
\begin{align}
\f{\abs{f(x+iy)}^2}{\abs{f(x-\delta)} \abs{f(x+\delta)}}
& \leq \biggl[ 1 + \f{y^2+1}{(\f{1}{100})}\bigg]^5 \prod_{n=1}^\infty
\biggl( 1 + \f{1+y^2}{n^2}\biggr)^2   \notag \\
& \leq C(1+y^{10}) \biggl( \f{\sinh \pi\sqrt{y^2+1}}{\pi \sqrt{y^2+1}}\biggr)^2 \lb{2.25d}
\end{align}
Thus, for any $\veps$, there is a $C_\veps$ with
\begin{equation} \lb{2.25e}
\abs{f(x+iy)} \leq C_\veps \exp ((\pi+\veps) \abs{y})
\end{equation}
for every $x+iy \in \C$, which is (\ref {2.20}).  This concludes the proof
of Theorem \ref {T2.1}.

\comm{
For $x \in \R$, choose $0 \leq \delta=\delta(x) \leq 1$ such
that $|(z-x_j)^2-\delta^2| \geq \frac {1} {100}$.

Notice that for every $y \in \R$ we
have
\begin{equation}
\frac {|f(x+iy)|^2} {|f(x+\delta) f(x-\delta)|}=\prod_{j \neq 0}
\biggl( 1+\frac {\delta^2+y^2} {|(z_j-x)^2-\delta^2|} \biggr).
\end{equation}
From the condition $|z_j-z_k| \geq |k-j|-1$, it follows that we can relabel
all $z_j$, except possibly $10$ of them, as $w_j$, $j \neq 0$, with $|w_j-x|
\geq |j|$.  Thus
\begin{align}
\frac {|f(x+iy)|^2} {|f(x+\delta) f(x-\delta)|} & \leq (101+100
y^2)^{10} \prod_{j=1}^\infty \biggl (1+\frac {1+y^2} {j^2} \biggr )\\
\nonumber
&\leq
(101+100 y^2)^{10} \biggl (\frac {\sinh(\pi y)} {\pi y} \biggr )^2.
\end{align}
}

\comm{
Thus, for $\abs{y}\geq 1$ and $y$ real,
\begin{align}
\abs{f(iy)}^2 &= \prod_{j\neq 0} \biggl( 1 + \f{y^2}{z_j^2}\biggr) \lb{2.27}  \\
&\leq (1+Cy^2)^2 \prod_{j=1}^\infty \biggl( 1+\f{y^2}{j^2}\biggr)^2 \lb{2.28} \\
&= (1+Cy^2)^2 \biggl( \f{\sinh (\pi y)}{\pi y}\biggr)^2 \lb{2.29}
\end{align}
\eqref{2.28}, with $C=\max(\f{1}{z_1^2}, \f{1}{z_{-1}^2})$, follows from $\abs{z_{j+1}} \geq \abs{j}$
and then \eqref{2.29} from the Euler product for $\sin (\pi z)$.
}

\begin{remark} It is possible to show, using the Phragm\'en--Lindel\"of principle \cite{Tit}, that if
one assumes, instead of (\ref {2.7}), the stronger $|f(z)| \leq C
e^{|z|^\delta}$, then it is possible
to weaken (\ref {2.6}) to
\begin{equation} \lb{2.26x}
|z_j| \geq |j|-1
\end{equation}
for if \eqref{2.26x} holds, then \eqref{2.26} implies that
\begin{equation} \lb{2.27x}
\abs{f(iy)} \leq C(1+\abs{y}) e^{\pi\abs{y}}
\end{equation}
Applying Phragm\'en--Lindel\"of to $(1-iz)^{-1} f(z) e^{i\pi z}$ on the sectors $\arg z\in [0,\pi/2]$
and $[\pi/2,\pi]$ proves that
\begin{equation} \lb{2.28x}
\abs{f(x+iy)} \leq C (1+\abs{z}) e^{\pi\abs{y}}
\end{equation}
\end{remark}

\section{Doing the Lubinsky Wiggle} \lb{s3}

Our goal in this section is to prove Theorem~2.

\begin{proof}[Proof of Theorem~2] By Egorov's theorem (see Rudin \cite[p.~73]{Rudin}), for every
$\veps$, there exists a compact set $\K\subset\Sigma$ with $\abs{\Sigma\setminus
\K} <\veps$ (with
$\abs{\cdot}=$ Lebesgue measure) so that on $\K$, $\f{1}{n+1} K_n(x,x)\equiv \tilde{q}_n(x)$ converges
uniformly to a limit we will call $\tilde{q}(x)$. If we prove that \eqref{1.26} holds for a.e.\ $x_0\in
\K$,
then by taking a sequence of $\veps$'s going to $0$, we get that \eqref{1.26} holds for a.e.\
$x_0\in\Sigma$

By Lebesgue's theorem on differentiability of integrals of $L^1$-functions (see Rudin \cite[Thm~7.7]{Rudin})
applied to the characteristic function of $\K$, for a.e.\ $x_0\in \K$,
\begin{equation} \lb{3.1}
\lim_{\delta\downarrow 0}\, (2\delta)^{-1} \abs{(x_0-\delta, x_0+\delta) \cap
\K} =1
\end{equation}
We will prove that \eqref{1.26} holds for all $x_0$ with \eqref{3.1} and with condition~(4).

$\f{1}{n+1} K_n(x+\f{a}{n} + \f{\bar z}{n}, x+\f{a}{n} + \f{z}{n})$ is analytic in $z$, so by a Cauchy
estimate and $a$ real,
\begin{align}
\biggl| \f{d}{da}\, \tilde{q}_n\biggl(x + \f{a}{n}\biggr)\biggr|
& \leq \sup_{\abs{z}\leq 1} \f{1}{n+1} \biggl| K_n\biggl( x+\f{a}{n} + \f{\bar z}{n}\, ,
x+\f{a}{n} + \f{z}{n}\biggr) \biggr| \notag \\
&= \sup_{\abs{z}\leq 1} \biggl| \tilde{q}_n \biggl( x+\f{a}{n} + \f{z}{n}\biggr)\biggr| \lb{3.2}
\end{align}
By a Schwarz inequality, for $x,y\in\bbC$,
\begin{equation} \lb{3.3}
\f{1}{n+1}\, \abs{K_n(x,y)} \leq (\tilde{q}_n(x) \tilde{q}_n(y))^{1/2}
\end{equation}

Thus, using the assumed \eqref{1.30}, for any $x_0$ for which \eqref{1.30} holds and any $A<\infty$, there
are $N_0$ and $C$ so for $n\geq N_0$,
\begin{equation} \lb{3.4}
\biggl| \tilde{q}_n \biggl( x_0 + \f{a}{n}\biggr) - \tilde{q}_n \biggl( x_0 + \f{b}{n}\biggr)\biggr| \leq C\abs{a-b}
\end{equation}
for all $a,b$ with $\abs{a}\leq A$, $\abs{b}\leq A$.

Since each $\tilde{q}_n$ is continuous and the convergence is uniform on $\K$, $\tilde{q}$ is continuous on
$\K$. Thus,
we have for each $A<\infty$,
\begin{equation} \lb{3.5}
\sup \biggl\{\bigg| \tilde{q} \biggl( x_0 + \f{a}{n}\biggr) - \tilde{q} (x_0)\biggr| \biggm| \abs{a}<A,\,
x_0 + \f{a}{n}\in \K\biggr\} \to 0
\end{equation}
as $n\to\infty$. By the uniform convergence,
\begin{equation} \lb{3.6x}
\sup \biggl\{\biggl| \tilde{q}_n \biggl( x_0 + \f{a}{n}\biggr) - \tilde{q}_n(x_0)\biggr| \biggm| \abs{a}<A,\,
x_0 + \f{a}{n}\in \K\biggr\} \to 0
\end{equation}

We next use the fact that \eqref{3.1} holds. It implies that
\begin{equation} \lb{3.6}
\sup_{\abs{b}\leq A}\, n\, \dist \biggl( x_0 + \f{b}{n}\, , \K\biggr) \to 0
\end{equation}
or equivalently, for any $\veps$, there is an $N_1$ so for $n\geq N_1$ and $\abs{b}< A$, there exists
$\abs{a} <A$ ($a$ will be $n$-dependent) so that $\abs{a-b}<\veps$ and $x_0 + \f{a}{n}\in
\K$. We have that
\begin{equation} \lb{3.7}
\biggl| \tilde{q}_n \biggl( x_0 + \f{b}{n}\biggr) - \tilde{q}_n(x_0)\biggr| \leq \biggl| \tilde{q}_n \biggl( x_0 + \f{b}{n}\biggr)
- \tilde{q}_n \biggl( x_0 + \f{a}{n}\biggr)\biggr| + \biggl| \tilde{q}_n \biggl( x_0 + \f{a}{n}\biggr) - \tilde{q}_n (x_0)\biggr|
\end{equation}
where $\abs{b-a}<\veps$ and $x_0 + \f{a}{n}\in \K$. By \eqref{3.4}, if $n\geq\max(N_0,N_1)$, the first
term is bounded by $C\veps$, and by \eqref{3.6}, the second term goes to zero, that is,
\begin{equation} \lb{3.8}
\sup_{\abs{b}<A}\, \biggl| \tilde{q}_n \biggl( x_0 + \f{b}{n}\biggr) - \tilde{q}_n(x_0)\biggr| \to 0
\end{equation}
Since $\tilde{q}_n(x_0)\to \tilde{q}(x_0)\neq 0$, we have
\begin{equation} \lb{3.9}
\sup_{\abs{b}<A}\, \biggl| \f{\tilde{q}_n(x_0 + \f{b}{n})}{\tilde{q}_n(x_0)} -1 \biggr| \to 0
\end{equation}
as $n\to\infty$, which is \eqref{1.26}.
\end{proof}

\section{Exponential Bounds for Perturbed Transfer Matrices} \lb{s4}

In this section, our goal is to prove Theorem~3. As noted in the introduction, our approach is
an extension of a theorem of Avila--Krikorian \cite[Lemma~3.1]{AvKr} exploiting that one can
avoid using cocycles and so go beyond the apparent limitation to ergodic situations. The
argument here is related to but somewhat different from variation of parameters techniques
(see, e.g., Jitomirskaya--Last \cite{JL} and Killip--Kiselev--Last \cite{KKL}) and should
have wide applicability.

\begin{proof}[Proof of Theorem~3] Fix $n$ and define for $j=1,2,\dots, n$,
\begin{align}
\ti A_j &= A_j \biggl( x_0 + \f{z}{n+1}\biggr) \lb{4.1} \\
A_j &= A_j(x_0) \lb{4.2} \\
T_j & = A_j \dots A_1 \qquad \ti T_j = \ti A_j \dots \ti A_1 \lb{4.3}
\end{align}
(Note that $\ti A_j$ and $\ti T_j$ are both $j$- and $n$-dependent.)

Note that, by \eqref{1.32},
\begin{equation} \lb{4.4}
\ti A_j - A_j = a_j^{-1} \begin{pmatrix}
\f{z}{n+1} & 0 \\
0 & 0 \end{pmatrix}
\end{equation}
so that
\begin{equation} \lb{4.5}
\norm{\ti A_j - A_j} \leq \alpha_-^{-1}\, \f{\abs{z}}{n+1}
\end{equation}

Write
\begin{align}
T_j^{-1} \ti T_j &= (T_j^{-1} \ti A_j T_{j-1})(T_{j-1}^{-1} \ti A_{j-1} T_{j-2})
\dots (T_1^{-1} \ti A_1 T_0) \lb{4.6} \\
&= (1+B_j) (1+B_{j-1}) \dots (1+B_1) \lb{4.7}
\end{align}
where
\begin{equation} \lb{4.8}
B_k = T_k^{-1} (\ti A_k -A_k) T_{k-1}
\end{equation}
Here we used
\begin{equation} \lb{4.9}
A_k T_{k-1} = T_k
\end{equation}

Since $T_k$ has determinant $1$ (see \eqref{1.34}), we have
\begin{equation} \lb{4.10}
\norm{T_k^{-1}} = \norm{T_k}
\end{equation}
So, by \eqref{4.5},
\begin{equation} \lb{4.11}
\norm{B_k} \leq \norm{T_k}\, \norm{T_{k-1}} \alpha_-^{-1}\, \f{\abs{z}}{n+1}
\end{equation}
Thus, since
\begin{equation} \lb{4.12}
\norm{1+B_j} \leq 1 + \norm{B_j} \leq \exp(\norm{B_j})
\end{equation}
\eqref{4.7} implies that
\begin{equation} \lb{4.13}
\norm{\ti T_j} \leq \norm{T_j} \exp\biggl( \alpha_-^{-1} \abs{z} \biggl[ \f{1}{n+1}
\sum_{k=1}^j \, \norm{T_k}\, \norm{T_{k-1}}\biggr]\biggr)
\end{equation}

By the Schwarz inequality, for $j=1,2,\dots, n$,
\begin{align}
\f{1}{n+1} \sum_{k=1}^j\, \norm{T_k}\, \norm{T_{k-1}}
&\leq \f{1}{n+1} \sum_{k=0}^j \, \norm{T_k}^2 \notag \\
&\leq \f{1}{n+1} \sum_{k=0}^n\, \norm{T_k}^2 \lb{4.14}
\end{align}
Using \eqref{1.39} and \eqref{4.13}, we find
\begin{equation} \lb{4.15}
\norm{\ti T_j} \leq \norm{T_j} \exp(C\alpha_-^{-1} \abs{z})
\end{equation}
This clearly holds for $j=0$ also. Squaring and summing,
\begin{equation} \lb{4.16}
\f{1}{n+1} \sum_{j=0}^n \, \norm{\ti T_j}^2 \leq \biggl( \f{1}{n+1} \sum_{j=0}^n \,
\norm{T_j}^2\biggr) \exp (2C\alpha_-^{-1} \abs{z})
\end{equation}
which is \eqref{1.40}.

Note that \eqref{1.41} implies \eqref{1.39} so that \eqref{1.42} is just \eqref{4.15}.
\end{proof}

We note that the argument above can also be used for more general
perturbative bounds. For example, suppose that
\begin{equation} \lb{4.16a}
C_1 \equiv \sup_n \norm{T_n(x_0)} <\infty
\end{equation}
for a given set of Jacobi parameters. Let $a'_n =a_n + \delta a_n$ and $b'_n =b_n + \delta b_n$ with
\begin{equation} \lb{4.17}
C_2 \equiv \sum_{n=1}^\infty\, \abs{\delta a_n} + \abs{\delta b_n} <\infty
\end{equation}
and
\begin{equation} \lb{4.18}
\alpha'_- =\inf\, a'_n >0
\end{equation}

Defining $\ti A_n, \ti T_n$ at energy $x_0$ but with $\{a'_n, b'_n\}_{n=1}^\infty$ Jacobi parameters,
one gets
\begin{equation} \lb{4.19}
\norm{\ti A_k -A_k} \leq C_3 [\alpha_-^{-1} + (\alpha'_-)^{-1}] (\abs{\delta a_k} + \abs{\delta b_k})
\end{equation}
for some universal constant $C_3$. Thus
\begin{equation} \lb{4.20}
\norm{B_k} \leq C_3 C_1^2 [\alpha_-^{-1} + (\alpha'_n)^{-1}] (\abs{\delta a_k} + \abs{\delta b_k})
\end{equation}
and
\begin{equation} \lb{4.21}
\norm{\ti T_n} \leq C_1 \exp (C_1^2 C_2 C_3 [\alpha_-^{-1} + (\alpha'_-)^{-1}])
\end{equation}
providing another proof of a standard $\ell^1$ perturbation result.

\section{Ergodic Jacobi Matrices and Ces\`aro Summability} \lb{s5}

In this section, our goal is to prove Theorem~4. We fix an ergodic Jacobi matrix setup.
We will need to use special solutions found by Deift--Simon in 1983:

\begin{theorem}[Deift--Simon \cite{S169}] \lb{T5.1} For any Jacobi matrix with $\Sigma_\ac (d\mu_\omega)$
{\rm{(}}which is a.e.\ $\omega$-independent{\rm{)}} of positive measure, for a.e.\ pairs  $(x,\omega)\in
\Sigma_\ac\times\Omega$ {\rm{(}}a.e.\ with respect to $dx\otimes d\eta(\omega)${\rm{)}}, there exist
sequences $\{u_n^\pm (x,\omega)\}_{n=-\infty}^\infty$ so that
\begin{equation} \lb{5.1}
T_n (x,\omega) \binom{u_1^\pm (x,\omega)}{a_0u_0^\pm(x,\omega)} =
\binom{u_{n+1}^\pm (x,\omega)}{a_n u_n^\pm (x,\omega)}
\end{equation}
with the following properties:
\begin{alignat}{2}
&\text{\rm{(i)}} \qquad && u_n^- (x,\omega) = \ol{u_n^+ (x,\omega)} \lb{5.2} \\
&\text{\rm{(ii)}} \qquad && a_n (u_{n+1}^+ u_n^- - u_{n+1}^- u_n^+) =-2i \lb{5.3} \\
&\text{\rm{(iii)}} \qquad && \abs{u_n^+ (x,\omega)} = \abs{u_0^+ (x, S^n \omega)} \lb{5.4} \\
&\text{\rm{(iv)}} \qquad && \int \abs{u_n^+ (x,\omega)}^2\, d\eta(\omega) <\infty \lb{5.5} \\
&\text{\rm{(v)}} \qquad && u_0^\pm \text{ is real} \lb{5.6}
\end{alignat}
\end{theorem}

Of course, by \eqref{5.4}, the integral in \eqref{5.5} is $n$-independent. For later purposes (see
Section~\ref{s6}), we will need an explicit formula for this integral. In fact, we will need explicit
formulae for $u_0, u_{-1}$ in terms of the $m$-function.

One defines for $\Ima z>0$, $\ti u_n^+ (z,\omega)$ to solve (i.e., \eqref{5.1})
\begin{equation} \lb{5.7}
a_n \ti u_{n+1}^+ + (b_n-z) \ti u_n^+ + a_{n-1} \ti u_{n-1}^+ =0
\end{equation}
with $\sum_{n=1}^\infty \abs{\ti u_n^+}^2 <\infty$. This determines $\ti u_n^+$ up to a constant, and so,
\begin{equation} \lb{5.8}
m(z,\omega) = -\f{\ti u_1^+ (z,\omega)}{a_0 \ti u_0^+ (z,\omega)}
\end{equation}
is normalization-independent and obeys, by \eqref{5.7},
\begin{equation} \lb{5.9}
m(z,\omega) = \f{1}{-z+b_1 -a_1^2 m(z,S\omega)}
\end{equation}
[{\it Note}: We have suppressed the $\omega$-dependence of $a_n,b_n$.]

As usual with solutions of \eqref{5.9},
\begin{equation} \lb{5.10}
m(z,\omega) = \int \f{d\mu_\omega^+ (x)}{x-z}
\end{equation}
where $d\mu_\omega^+$ is the measure associated to the half-line Jacobi matrix, $J_\omega$.

For a.e.\ $x\in\Sigma_\ac$ and a.e.\ $\omega$, $m(x+i0,\omega)$ exists and has
\begin{equation} \lb{5.11}
\Ima m(x+i0, \omega) >0 \qquad (\text{a.e. } x\in\Sigma_\ac)
\end{equation}
We normalize the solution $u^+$ obeying Theorem~\ref{T5.1} by defining:
\begin{align}
u_0^+ (x,\omega) &= \f{1}{a_0 [\Ima m(x+i0,\omega)]^{1/2}} \lb{5.12} \\
u_1^+ (x,\omega) &= -\f{m(x+i0,\omega)}{[\Ima m(x+i0,\omega)]^{1/2}} \lb{5.13}
\end{align}
(We have listed all the formulae because \cite{S169} only consider the case $a_n\equiv 1$.)
$u_n^+$ are then determined by the difference equation and $u_n^-$ by \eqref{5.2}.

Of course, we have
\begin{equation} \lb{5.14}
p_n = \f{u_{n+1}^+ - u_{n+1}^-}{u_1^+ - u_1^-}
\end{equation}
since both sides obey the same difference equations with $p_{-1}=0$ (since $u_0^+ = u_0^-$) and $p_0=1$.

By \eqref{5.14}, to prove Theorem~4 we need to show
\begin{equation} \lb{5.14a}
\f{1}{n}\, \sum_{j=0}^{n-1} (u_{j+1}^+ - u_{j+1}^-)^2
\end{equation}
exists. This follows from the existence of
\begin{equation} \lb{5.14b}
\lim_{n\to\infty}\, \f{1}{n}\, \sum_{j=1}^n \, \abs{u_j^+}^2
\end{equation}
and
\begin{equation} \lb{5.14c}
\lim_{n\to\infty}\, \f{1}{n}\, \sum_{j=1}^n (u_j^+)^2
\end{equation}

From \eqref{5.4} and the ergodic theorem (plus \eqref{5.5}), the a.e.\ $\omega$ existence of the limit
in \eqref{5.14b} is immediate. In cases like the almost Mathieu equation with Diophantine frequencies
where $u_n^+$ is almost periodic, one also gets the existence of the limit in \eqref{5.14c} directly,
but there are examples, like the almost Mathieu equation with frequencies whose dual has singular
continuous spectrum, where the phase of $u_n^+$ is not almost periodic. So this argument does not work in
general. In fact, we will eventually prove that for a.e.\ $(x,\omega)$ in $\Sigma_\ac \times \Omega$
(see Theorem~\ref{T6.3})
\begin{equation} \lb{5.14d}
\lim_{n\to\infty}\, \f{1}{n}\, \sum_{j=1}^n (u_j^+)^2 =0
\end{equation}
It would be interesting to have a direct proof of this (for the periodic case, see \cite{Rice}) rather
than the indirect path we will take.

Define the $2\times 2$ matrix
\begin{equation} \lb{5.15}
U_n(x,\omega) = \f{1}{(-2i)^{1/2}} \begin{pmatrix}
u_{n+1}^+ (x,\omega) & u_{n+1}^- (x,\omega) \\
a_n u_n^+ (x,\omega) & a_n u_n^- (x,\omega)
\end{pmatrix}
\end{equation}
(where we fix once and for all a choice of $\sqrt{-2i}$). By \eqref{5.3},
\begin{equation} \lb{5.16}
\det(U_n(x,\omega)) =1
\end{equation}
and, by \eqref{5.1},
\begin{equation} \lb{5.17}
T_n (x,\omega) U_0(x,\omega) = U_n(x,\omega)
\end{equation}
or
\begin{equation} \lb{5.18}
T_n(x,\omega) = U_n(x,\omega) U_0 (x,\omega)^{-1}
\end{equation}

For now, we fix $x\in\Sigma_\ac$ with
\begin{equation} \lb{5.19}
E([a_0(\omega)^2 \Ima m(x+i0,\omega)]^{-1}) <\infty
\end{equation}
(known Lebesgue a.e.\ by Kotani theory; see \cite{S168,S169}), so $U_n$ can be defined and is in $L^2$.
We are heading towards a proof of

\begin{theorem}\lb{T5.2} For any fixed matrix, $Q$, a.e.\ $\omega$, as matrices
\begin{equation} \lb{5.20}
\lim_{n\to\infty}\, \f{1}{n}\, \sum_{j=0}^{n-1} T_j(x,\omega)^t QT_j(x,\omega)
\end{equation}
exists.
\end{theorem}

\begin{proof}[Proof of Theorem~4 given Theorem~\ref{T5.2}] Pick $Q=\left(\begin{smallmatrix}
1&0 \\ 0&0 \end{smallmatrix}\right)$. Then the $1,1$ matrix element of $T_j(x,\omega)^t QT_j
(x,\omega)$ is $p_j(x,\omega)^2$, so \eqref{1.45} holds. Similarly, the $2,2$ matrix element is
$q_j(x,\omega)^2$.
\end{proof}

\eqref{5.18} plus \eqref{5.5} will imply critical a priori bounds on $\norm{T_n(x,\dott)}_{L^1(d\eta)}$.
It will be convenient to use the Hilbert--Schmidt norm on these $2\times 2$ matrices.

\begin{lemma} \lb{L5.3} We have
\begin{equation} \lb{5.21}
\sup_n \int \norm{T_n (x,\omega)}\, d\eta(\omega) <\infty
\end{equation}
\end{lemma}

\begin{proof} Since $\det(U_n) =1$,
\begin{equation} \lb{5.22}
\norm{U_n(x,\omega)^{-1}} = \norm{U_n(x,\omega)}
\end{equation}
Thus, by \eqref{5.18},
\begin{equation} \lb{5.23}
\norm{T_n(x,\omega)} \leq \norm{U_n(x,\omega)}\, \norm{U_0(x,\omega)}
\end{equation}
By the Schwarz inequality,
\begin{align*}
\sup_n \int \norm{T_n(x,\omega)} \, d\eta(\omega)
&\leq \sup_n \int \norm{U_n(x,\omega)}^2 \, d\eta(\omega) \\
&= \int \norm{U_0(x,\omega)}^2 \, d\eta(\omega) \\
&< \infty
\end{align*}
by \eqref{5.5} and the fact that since \eqref{5.4} holds and we use Hilbert--Schmidt norms,
\begin{equation} \lb{5.24}
\norm{U_j (x,\omega)} = \norm{U_0(x,S^j\omega)}
\end{equation}
\end{proof}

Let $A_j(\omega)$ be the matrix \eqref{1.32} with $a_j =a_j(\omega)$, $b_j=b_j(\omega)$ and let
\begin{equation} \lb{5.25a}
A(\omega) \equiv A_1(\omega)
\end{equation}
so
\begin{equation} \lb{5.25b}
A_j(\omega) = A(S^{j-1} \omega)
\end{equation}
and the transfer matrix for $J_\omega$ is
\begin{equation} \lb{5.25c}
T_n(\omega) = A(S^{n-1}\omega) \dots A(\omega)
\end{equation}

Now form the suspension
\begin{equation} \lb{5.26}
\widehat\Omega =\Omega\times \bbS\bbL(2,\bbC)
\end{equation}
and define $\widehat S\colon\widehat\Omega \to\widehat\Omega$ by
\begin{equation} \lb{5.27}
\widehat S(\omega,C) = (S\omega, A(\omega) C)
\end{equation}
so
\begin{equation} \lb{5.28}
\widehat S^n(\omega,C) = (S^n \omega, T_n(\omega)C)
\end{equation}

\begin{theorem} \lb{T5.4} There exists an $\widehat S$-invariant probability measure, $d\nu$,
on $\widehat\Omega$ whose projection onto $\Omega$ is $d\eta$ and with
\begin{equation} \lb{5.29}
\int \norm{C}\, d\nu(\omega,C) <\infty
\end{equation}
\end{theorem}

\begin{proof} Pick any probability measure $\mu_0$ on $\bbS\bbL(2,\bbC)$ with $\int \norm{C}^k\, d\mu_0
(C) <\infty$ for all $k$. For example, one could take $d\mu_0(C)=Ne^{-\norm{C}^2} d\,\text{Haar}(C)$
where $N$ is a normalization constant. Let $\widehat S_*$ be induced on measures on $\widehat\Omega$
by $[\widehat S_*(\nu)](f) =\nu(f\circ\widehat S)$. Let
\begin{equation} \lb{5.30}
\nu_n = \widehat S^n_* (\eta\otimes \mu_0)
\end{equation}

Then the invariance of $\eta$ under $S_*$ implies the projection of $\nu_n$ is $\eta$ and
\begin{align}
\int \norm{C}\, d\nu_n
&= \int \norm{T_n(\omega)C}\, d\eta\otimes d\mu_0 \notag \\
&\leq \biggl( \int \norm{T_n(\omega)}\, d\eta\biggr) \biggl( \int \norm{C}\, d\mu_0\biggr) \lb{5.31}
\end{align}
which, by \eqref{5.21}, is uniformly bounded in $n$.

Let $\ti\nu_n$ be the Ces\`aro averages of $\nu_n$, that is,
\begin{equation} \lb{5.32}
\ti\nu_n =\f{1}{n}\, \sum_{j=0}^{n-1} \nu_j
\end{equation}
So, by \eqref{5.31},
\begin{equation} \lb{5.33}
\sup_n \int \norm{C}\, d\ti\nu_n <\infty
\end{equation}
so $\{\ti\nu_n\}$ are tight, that is,
\[
\lim_{K\to\infty}\, \sup_n\, \ti\nu_n \{C\mid \norm{C}\geq K\} \to 0
\]
which implies that $\ti\nu_n$ has a weak limit point in probability measures on $\wti\Omega$. This
weak limit point is invariant and, by \eqref{5.33}, it obeys \eqref{5.29}.
\end{proof}

\begin{lemma}\lb{L5.5}
Let $L<\infty$. Let
\begin{equation} \lb{5.33a}
\widehat\Omega_L =\{(\omega,C)\mid \norm{U_0(\omega)} < L, \, \norm{C} <L\}
\end{equation}
Then for any $\veps$, there is a $K$ so that for a.e.\ $(\omega,C)\in\widehat\Omega_L$,
\begin{equation} \lb{5.33b}
\lim_{n\to\infty}\, \f{1}{n}\, \sum_{\substack{ j\in B(K,\omega,C) \\ 0\leq j
\leq n-1}}
\norm{T_j(\omega)C}^2 \leq \veps
\end{equation}
where
\begin{equation} \lb{5.33c}
B(K,\omega,C) =\{j\mid\norm{T_j(\omega)C} \geq K\}
\end{equation}
\end{lemma}

\begin{proof} Since $U_0(\omega)\in L^2 (d\eta)$, we have
\begin{equation} \lb{5.33d}
\lim_{s\to\infty}\, \int_{\norm{U_0(\omega)}\geq s} \norm{U_0(\omega)}^2 \, d\eta(\omega) =0
\end{equation}
so for any $\delta >0$, there exists $s(\delta)$ so that the integral is less than $\delta$.

Let $\wti B(\wti K,\omega)$ be defined by
\begin{equation} \lb{5.33e}
\wti B(\wti K,\omega) = \{j\mid \norm{U_j (\omega)} \geq \wti K\}
\end{equation}
By the Birkhoff ergodic theorem and \eqref{5.24} for a.e.\ $\omega$,
\begin{equation} \lb{5.33f}
\lim_{n\to\infty}\, \f{1}{n} \sum_{\substack{ j\in\wti B(\wti K,\omega) \\ 0\leq j
\leq n-1}}
\norm{U_j(\omega)}^2=\int_{\norm{U_0(\omega)}\geq\wti K} \norm{U_0(\omega)}^2 \, d\eta\leq\delta
\end{equation}
if $\wti K\geq s(\delta)$.

Given $\veps$ and $L$, let $\delta = \veps/L^2$ and $K\geq L^2 s(\delta)$. Since
\begin{equation} \lb{5.33g}
\norm{T_j(\omega)C}\leq \norm{U_j(\omega)} L^2
\end{equation}
if $(\omega,C)\subset \Omega_L$,
\[
B(K,\omega,C) \subset \wti B\biggl( \f{K}{L^2}\, ,\omega\biggr)
\]
So, by \eqref{5.33f} and \eqref{5.33g},
\begin{equation} \lb{5.33h}
\lim_{n\to\infty}\, \f{1}{n} \sum_{\substack{ j\in B(K,\omega,C) \\ 0\leq j
\leq n-1}}
\norm{T_j(\omega) C}^2 \leq L^2 \delta =\veps
\end{equation}
which is \eqref{5.33a}.
\end{proof}

\begin{proof}[Proof of Theorem~\ref{T5.2}] Without loss, suppose $\norm{Q}\leq 1$. Define on
$\widehat\Omega$
\begin{equation} \lb{5.34}
f_n(\omega,C) =\f{1}{n}\, \sum_{j=0}^{n-1} C^t T_j(x,\omega)^t QT_j (x,\omega) C
\end{equation}
If we prove that this has a pointwise limit for $\nu$ a.e.\ $(\omega,C)$, we are done: since $\eta$ is
the projection of $\nu$, for $\eta$ a.e.\ $\omega$, there are some $C$ for which \eqref{5.34} has a
limit. But $C$ is invertible, so $(C^t)^{-1} f_n C^{-1}$ has a limit, that is, \eqref{5.20} does.

Notice that if
\begin{equation} \lb{5.35}
h(\omega,C) =C^t QC
\end{equation}
then $f_n(\omega,C)$ is a Ces\`aro average of $h(\widehat S^j(\omega,C))$, so we can almost use the ergodic
theorem except we only know a priori that $\int \norm{h(\omega,C)}^{1/2}\, d\nu <\infty$, not $\int
\norm{h(\omega,C)}\, d\nu <\infty$, so we need to use Lemma~\ref {L5.5}.

Fix $L$ and consider $(\omega,C)\in \widehat \Omega_L$. Let
\begin{equation} \lb{5.36}
h_K(\omega,C) = \begin{cases}
C^t QC &\text{if } \norm{C}\leq K \\
0 & \text{if } \norm{C} >K
\end{cases}
\end{equation}
Then, since $\norm{Q} \leq 1$,
\begin{equation} \lb{5.37}
\norm{h_K(\widehat S^j(\omega,C)) - h(\widehat S^j(\omega,C))} \leq
\begin{cases}
0 & \text{if } j\notin B(K,\omega,C) \\
\norm{T_j(\omega)C}^2 & \text{if } j\in B(K,\omega,C)
\end{cases}
\end{equation}
It follows that if
\begin{equation} \lb{5.38}
f_n^{(K)} (\omega,C) =\f{1}{n}\, \sum_{j=0}^{n-1} h_K (\widehat S^j(\omega,C))
\end{equation}
then
\[
\norm{f_n^{(K)}(\omega,C) -f_n(\omega,C)} \leq \text{sum on left side of \eqref{5.33b}}
\]
So, by Lemma~\ref{L5.5},
\begin{equation} \lb{5.39}
\limsup_{n\to\infty}\, \norm{f_n^{(K)}(\omega,C)-f_n(\omega,C)}\leq \veps
\end{equation}
if
\begin{equation} \lb{5.40}
K\geq K(\veps,L)
\end{equation}
given by the lemma.

For any finite $K$, $h_K$ is bounded, so the Birkhoff ergodic theorem and the invariance of $\nu$ imply, for
a.e.\ $(\omega,C)$, $\lim f_n^{(K)}(\omega,C)$ exists.  Thus (\ref {5.39})
and (\ref {5.40}) imply that $\lim f_n^{(K)}(\omega,C)$ forms a Cauchy
sequence as $K \to \infty$ (among, say, integer values), and that its limit
is also $\lim f_n(\omega,C)$, for a.e. $(\omega,C) \in \widehat
\Omega_L$.

\comm{
On $\widehat\Omega_L$,
\[
\norm{f_n(\omega,C)} \leq L^4 \, \f{1}{n}\, \sum_{j=0}^{n-1} \, \norm{U_j(\omega)}^2
\]
has the limit $L^4\int \norm{U_0(\omega)}\, d\eta(\omega)$ for a.e.\ $\omega$, so the $\norm{f_n(\omega,C)}$
are bounded as $n\to\infty$ for a.e.\ $(\omega,C)\in\widehat\Omega_L$.

For any finite $K$, $h_K$ is bounded, so the Birkhoff ergodic theorem and the invariance of $\nu$ imply, for
a.e.\ $(\omega,C)$, $\lim f_n^{(K)}(\omega,C)$ exists. It follows that any two limit points of $f_n(\omega,C)$
lie within $2\veps$ of each other. Since $\veps$ is arbitrary and $\norm{f_n(\omega,C)}$ is bounded, the limit
exists for a.e.\ $(\omega,C)\in\widehat\Omega_L$.
}

Since $L$ is arbitrary and $\nu (\widehat\Omega\setminus\widehat\Omega_L)\to 0$ on account of
$\int \norm{U_0(\omega)}^2\, d\nu<\infty$, we see that $f_n$ has a limit for
a.e.\ $\omega,C$.
\end{proof}

\section{Equality of the Local and Microlocal DOS} \lb{s6}

Our main goal in this section is to prove Theorem~5. We know from Theorem~4 that for a.e.\ $\omega\in\Omega$
and $x_0\in\Sigma_\ac$, we have that
\begin{equation} \lb{6.1}
\f{1}{n+1}\, K_n (x_0,x_0) \to k_\omega (x_0)
\end{equation}
some positive function. By Theorems~1 and 2, this implies that the spacing of zeros at a.e.\
Lebesgue point is
\begin{equation} \lb{6.2}
x_{j+1}^{(n)}(x_0) - x_j^{(n)}(x_0) \sim \f{1}{nw_\omega(x_0) k_{\omega}(x_0)}
\end{equation}

Thus, for fixed $K$ large, in an interval $(x_0-\f{K}{n}, x_0+\f{K}{n})$, the number of zeros is $2K w(x_0)
k(x_0)$. On the other hand, if $\rho_\infty(x_0)$ is the density of states, for a.e.\ $x_0$ in the
a.c.\ part of the support of $d\nu_\infty$, the number of zeros in $(x_0-\delta, x_0+\delta)$ is
approximately $2\delta n\rho(x_0)$. If $\delta$ were $K/n$, this would tell us that
\begin{equation} \lb{6.3}
w_\omega(x_0) k_\omega(x_0) =\rho_\infty (x_0)
\end{equation}
which is precisely \eqref{1.23}.

Of course, $\rho_\infty$ is defined by first taking $n\to\infty$ and then $\delta\downarrow 0$, so we
cannot set $\delta=K/n$, but \eqref{6.3} is an equality of a local density of zeros obtained by taking
intervals with $O(n)$ zeros as $n\to\infty$ and a microlocal individual spacing as in \eqref{6.2}.

So define
\begin{equation} \lb{6.4}
\rho_L (x_0,\omega) =w_\omega(x_0) k_\omega(x_0)
\end{equation}
the microlocal DOS. Notice that we have indicated an $\omega$-dependence of $\rho_L$ because, at this point,
we have not proven $\omega$-independence. $\omega$-independence often comes from the ergodic theorem---we
determined the existence of $k_\omega(x_0)$ using the ergodic theorem, but unlike for $\rho_\infty$, the
underlying measure was only invariant, not ergodic, and indeed, $k_\omega$, the object we controlled is
{\it not\/} $\omega$-independent.

Of course, once we prove $\rho_L=\rho_\infty$, $\rho_L$ will be proven $\omega$-independent, but we will,
in fact, go the other way: we first prove that $\rho_L$ is $\omega$-independent, use that to show that
if $u$ is the Deift--Simon wave function, then the average of $u^2$ (not $\abs{u}^2$) is zero, and
use that to prove that $\rho_L=\rho_\infty$.

\begin{theorem}\lb{T6.1} Suppose that $J_\omega$ is a family of ergodic Jacobi matrices. Let $\rho_L(x,\omega)$
be given by \eqref{6.1}/\eqref{6.4} for $x\in\Sigma_\ac$, $\omega\in\Omega$. Then for a.e.\ $x\in\Sigma_\ac$,
$\rho_L(x,\omega)$ is a.e.\ $\omega$-independent.
\end{theorem}

\begin{proof} Since $\rho_L (x,\omega)$ is jointly measurable for $(x,\omega)\in\Sigma_\ac\times\Omega$,
$\rho_L(x,\dott)$ is measurable for a.e.\ $x$. Since $S$ is ergodic, it suffices to prove that $\rho_L
(x,S\omega)=\rho_L(x,\omega)$ for a.e.\ $(x,\omega)$.

Let $p_n (x,\omega)$ be the OPs for $J_\omega$. Then the zeros of $p_{n-1}(x,S\omega)$ and $p_n(x,\omega)$
interlace. It follows for any interval $[x_0-\f{A}{n}, x_0+\f{A}{n}] =I_{n,A}(x_0)$,
\begin{equation} \lb{6.5}
\abs{\#\text{ of zeros of } p_n(x,\omega) \text{ in } I_{n,A}(x_0) - \# \text{ of zeros of } p_{n-1}
(x,S\omega) \text{ in } I_{n,A}(x_0)} \leq 2
\end{equation}
If $\rho_L(x_0,S\omega)\neq\rho_L(x_0,\omega)$ and $A=k\rho_L(x_0,\omega)^{-1}$ with $k$ large, it is easy to
get a contradiction between \eqref{6.5} and \eqref{6.2}. Thus, $\rho_L(x,\omega)=\rho_L(x,S\omega)$
as claimed.
\end{proof}

Next, we need a connection between $\rho_L$ and $u$. Recall (see \eqref{5.14})
\begin{equation} \lb{6.6}
p_n(x,\omega) = \f{\Ima u_{n+1}^+ (x,\omega)}{\Ima u_1^+(x,\omega)}
\end{equation}
and, by \eqref{5.13},
\begin{equation} \lb{6.7}
\Ima u_1^+(x,\omega) = -[\Ima m(x+i0,\omega)]^{1/2}
\end{equation}
and, by \eqref{5.10}, for a.e.\ $x\in\Sigma_\ac$,
\begin{equation} \lb{6.8}
\Ima m(x+i0,\omega)=\pi w_\omega(x)
\end{equation}

Thus, if we define
\begin{equation} \lb{6.9}
\Av_\omega(f_j(\omega)) \equiv \lim_{n\to\infty}\, \f{1}{n}\, \sum_{j=1}^n f_j(\omega)
\end{equation}
then
\begin{equation} \lb{6.10}
\rho_L(x,\omega) = \f{1}{\pi}\, \Av_\omega ([\Ima u_j^+ (x,\omega)]^2)
\end{equation}

Note that $\Ima u_j^+(x,\omega)$ is not $\Ima u_0^+ (x,S^j \omega)$, so we cannot write \eqref{6.10}
as an integral. In fact, the $\omega$-independence of the right side of \eqref{6.10} (because of
$\omega$-independence of the left side) will have important consequences.

To see where we are heading, we note the following result of Kotani \cite{Kot97}; see Damanik
\cite[Thm.~5]{Dam}:

\begin{theorem}[Kotani \cite{Kot97}]\lb{t6.2} For a.e.\ $x\in\Sigma_\ac$,
\begin{equation} \lb{6.11}
\rho_\infty(x) = \f{1}{2\pi} \int \abs{u_0^+(x,\omega)}^2\, d\eta(x)
\end{equation}
\end{theorem}

\begin{remarks} 1. \cite{Kot97,Dam} treat $a_n\equiv 1$, but it is easy to accommodate general $a_n$.

\smallskip
2. Kotani's theorem is not stated in this form but rather as (see eqn.~(22) in Damanik \cite{Dam})
\begin{equation} \lb{6.12}
\pi\rho_\infty (x) = \int \Ima G_\omega(0,0; x+i0)\, d\eta(\omega)
\end{equation}
where $G_\omega$ is the whole-line Green's function. Because $G_\omega$ is reflectionless, $G_\omega$
is pure imaginary and
\begin{align}
\Ima (G_\omega (0,0; x+i0)) &= [2a_0^2 \Ima m(x+i0,\omega)]^{-1} \lb{6.13} \\
&= \tfrac12\, \abs{u_0^+ (x,\omega)}^2 \lb{6.14}
\end{align}
by \eqref{5.12}.
\end{remarks}

Thus, the key to proving $\rho_L =\rho_\infty$ will be to show that
\begin{equation} \lb{6.15}
\Av_\omega ([\Ima u_j^+ (x,\omega)]^2) = \Av_\omega([\Real u_j^+ (x,\omega)]^2)
\end{equation}
Note that \eqref{6.10} includes that the $\Av_\omega ([\Ima u_j^+]^2)$ exists and, by the ergodic theorem,
$\Av_\omega(\abs{u_j^+}^2)$ exists, so we know for a.e.\ $(x,\omega)\in\Sigma_\ac\times\Omega$ that
$\Av_\omega([\Real u_j^+ (x,\omega)]^2)$ exists. We are heading towards:

\begin{theorem}\lb{T6.3} Suppose $x\in\Sigma_\ac$ is such that $\rho_L(x,\omega)$ exists for a.e.\ $\omega$
and is $\omega$-independent, and that
\begin{equation} \lb{6.16}
\nu_\infty ((-\infty,x])\neq \tfrac12
\end{equation}
Then for a.e.\ $\omega$,
\begin{equation} \lb{6.17}
\Av_\omega((u_j^+(x,\omega))^2) =0
\end{equation}
\end{theorem}

\begin{proof}[Proof of Theorem~5 given Theorem~\ref{T6.3}] \eqref{6.16} fails at at most a single $x$ in
$\Sigma_\ac$, so \eqref{6.17} holds for a.e.\ $(x,\omega)\in\Sigma_\ac\times\Omega$. Its real part implies
\eqref{6.15}, and so for a.e.\ $(x,\omega)$,
\begin{align}
\Av_\omega([\Ima u_j^+ (x,\omega)]^2) &= \tfrac12\, \Av_\omega(\abs{u_j^+ (x,\omega)}^2) \lb{6.18} \\
&= \tfrac12 \int \abs{u_0^+ (x,\omega)}^2\, d\eta(x) \lb{6.19}
\end{align}
by the ergodic theorem. By \eqref{6.10}, \eqref{6.11}, and the definition \eqref{6.1}/\eqref{6.4} of $\rho_L$,
we see that the limit in \eqref{1.45} is $\rho_\infty (x)/w_\omega(x)$.
\end{proof}

\begin{proof}[Proof of Theorem~\ref{T6.3}] Fix $x\in\Sigma_\ac$ (at each stage, we work up to sets of
Lebesgue measure $0$). Define $\varphi(\omega)\in (0,2\pi)$ by
\begin{equation} \lb{6.20}
\Arg (-m(x+i0,\omega)) = -\varphi(\omega)
\end{equation}
Then $\varphi(\omega)\in (0,\pi)$ by $\Ima m>0$. Let ($\varphi$ and $s_n$ also depend on $x$)
\begin{equation} \lb{6.21}
s_n(\omega) =\sum_{j=1}^n \varphi (S^{j-1}\omega)
\end{equation}
Then by \eqref{5.8} and \eqref{5.4},
\begin{equation} \lb{6.22}
u_n^+ (x,\omega) =e^{-is_n(\omega)} u_0^+ (x,S^n \omega)
\end{equation}
and
\begin{equation} \lb{6.23}
u_{n+j}^+ (x,\omega) = e^{-is_n(\omega)} u_j^+ (x,S^n \omega)
\end{equation}
It follows that for each fixed $n$,
\begin{equation} \lb{6.24}
\Av_\omega ((\Ima u_j^+((x,S^n\omega))^2) =\Av_\omega ((\Ima e^{is_n(\omega)} u_j^+ (x,\omega))^2)
\end{equation}

If $s,x,y$ are real,
\begin{align}
(\Ima (e^{is}(x+iy)))^2 &= (x\sin s + y\cos s)^2 \notag \\
&= y^2 + (\sin^2 s)(x^2 -y^2) + xy (\sin 2s) \lb{6.25}
\end{align}
and thus,
\begin{equation} \lb{6.26}
\begin{split}
\text{LHS of \eqref{6.24}} &= \Av_\omega ([\Ima (u_j^+ (x,\omega))]^2) + \sin^2 s_n(\omega) R(\omega)   \\
& \qquad\qquad +\tfrac12\, \sin (2s_n(\omega)) I(\omega)
\end{split}
\end{equation}
where
\begin{align}
R(\omega) &= \Av_\omega (\Real ((u_j^+ (x,\omega))^2)) \lb{6.27} \\
I(\omega) &= \Av_\omega (\Ima ((u_j^+ (x,\omega))^2)) \lb{6.28}
\end{align}
(all such averages having been previously shown to exist).

\comm{
We already know that average in $R(\omega)$ exists, and if the average
on the LHS of \eqref{6.25} exists,
either $I(\omega)$ exists or $\sin (2s_n(\omega))=0$.
}

We know that for a.e.\ $(x,\omega)$, for $n=0,1,2, \dots$, LHS of \eqref{6.24} exists and is $n$-independent
(and equal to $\rho_L (x,\omega)$). For such $(x,\omega)$, \eqref{6.26} implies that for all $n$,
\begin{equation} \lb{6.29}
\sin s_n(\omega) [\sin s_n(\omega) R(\omega) + \cos s_n(\omega) I(\omega)] =0
\end{equation}

We want to consider two cases:
\begin{SL}
\item[{\it Case 1.}] For a positive measure set of $\omega$,
\begin{equation} \lb{6.30}
s_2(\omega) = \pi \qquad s_4(\omega) =2\pi \qquad s_6(\omega) = 3\pi \qquad \dots
\end{equation}

\item[{\it Case 2.}] For a.e.\ $\omega$, there is an $n(\omega)$ so
\begin{equation} \lb{6.31}
s_{2j}(\omega) = j\pi \quad (j=1, \dots, n-1) \qquad s_{2n}(\omega) \neq n\pi
\end{equation}
\end{SL}

In Case~1, for such $\omega$, we have $\frac {s_n(\omega)} {n \pi} \to
\frac {1} {2}$.  It follows by standard Sturm oscillation theory
(see, e.g., \cite{JoMo}) that $\frac {s_n(\omega)} {n \pi} \to
\nu_\infty ((-\infty, x])$ for almost every $\omega$.
Thus, the hypothesis \eqref{6.16}
eliminates Case~1.

\comm{
$\Ima (u_{2j}^+ (x,\omega)) =0$, that is,
\begin{equation} \lb{6.32}
p_1 (x,\omega) = p_3 (x,\omega) = \cdots = 0
\end{equation}
So the overall density of zeros of $\rho$ is $\f12$ and, by standard Sturm oscillation arguments
(see, e.g., \cite{JoMo}) and the fact that this holds for a positive measure set of $\omega$'s,
we have $\rho_\infty ((-\infty, x])=\f12$. Thus, the hypothesis \eqref{6.16}
eliminates Case~1.\marginpar{I found it difficult to parse the statement
about overal density of zeros of $\rho$.  Should we talk perhaps
about sign changes in the sequence $p_i(x,\omega)$ instead?}
}

For Case~2, suppose first that $n$ is odd, so $s_{2(n-1)}(\omega)$ is a multiple of $2\pi$ and \eqref{6.21} for
$2n-1$ and $2n$ imply
\begin{gather}
\sin(\varphi_{2n-1}) [\sin (\varphi_{2n-1}) R+\cos(\varphi_{2n-1}) I] = 0 \lb{6.33} \\
\sin(\varphi_{2n-1} + \varphi_{2n}) [\sin (\varphi_{2n-1} + \varphi_{2n}) R + \cos (\varphi_{2n-1}
+\varphi_{2n}) I]  = 0 \lb{6.34}
\end{gather}
Since $\varphi_{2n-1}\in (0,\pi)$, $\sin(\varphi_{2n-1}) \neq 0$ and since $\varphi_{2n-1} + \varphi_{2n}
\in (0,2\pi)\setminus \{\pi\}$, (for if it equals $\pi$, then $s_{2n} =n\pi$!), $\sin(\varphi_{2n-1}
+ \varphi_{2n}) \neq 0$.

The determinant of equations \eqref{6.33}/\eqref{6.34} is
\begin{equation} \lb{6.35}
-\sin(\varphi_{2n-1}) \sin(\varphi_{2n-1} + \varphi_{2n}) \sin (\varphi_{2n}) \neq 0
\end{equation}
since
\begin{equation} \lb{6.36}
\sin(A) \cos(B) -\sin(B) \cos(A) =\sin(A-B)
\end{equation}
Here $\neq 0$ in \eqref{6.35} comes from $\varphi_{2n}\in (0,\pi)$, so $\sin(\varphi_{2n})\neq 0$.

The nonzero determinant means that \eqref{6.33}/\eqref{6.34} $\Rightarrow I=R=0$, that is, $\Av_\omega
((u_j^+)^2) =0$ for a.e.\ $\omega$. If $n$ is even, $s_{2(n-1)}(\omega)$ is an odd multiple of $\pi$
and all equations pick up minus signs, so the argument is unchanged.
\end{proof}

\section{Assorted Remarks} \lb{s7}

1. \ We have proven for general ergodic Jacobi matrices that for a.e.\ $(x,\omega)\in\Sigma_\ac\times\Omega$,
\begin{equation} \lb{7.1}
\f{1}{n+1}\, K_n(x,x;\omega) \to \f{\rho_\infty(x)}{w_\omega (x)}
\end{equation}
Here $\rho_\infty$ is the Radon--Nikodym derivative of the a.c.\ part of $d\rho_\infty$. Based on
\cite{MNT91,Tot}, where results of this type are proven for regular measures, one expects
\begin{equation} \lb{7.2}
\rho_\infty(x)=\rho_\fre(x)
\end{equation}
Here $\fre$ is the essential spectrum of $J_\omega$ and $\rho_\fre$ its equilibrium measure. In
\cite{EqMC}, it is proven (see Thm.~1.15 there)

\begin{theorem}\lb{T7.1} If $\Sigma_\ac$ is not empty, then \eqref{7.2} holds if and only if, for
$\rho_\fre$ a.e.\ $x$,
\begin{equation} \lb{7.3}
\gamma(x)=0
\end{equation}
\end{theorem}

In particular, for examples where \eqref{7.3} fails on a set of positive Lebesgue measure in $\fre$
(e.g., \cite{Bour,Bour2002,FK2005,FK2006}), \eqref{7.2} may not hold. On the other hand, for examples
like the almost Mathieu equation where it is known that \eqref{7.3} holds on all of $\fre$ (see \cite{BJ}),
\eqref{7.2} holds. The moral is that \eqref{7.2} holds some, but not all, of the time for ergodic Jacobi
matrices.

\smallskip
2. \ Here is an interesting example that provides a deterministic problem where one has strong clock behavior
but with a density of zeros, $\rho_\infty$, which is not $\rho_\fre$. Let $d\mu$ be a measure on $[-2,2]$
of the form ($N$ is a normalization constant)
\begin{equation} \lb{7.4}
d\mu(x) = N^{-1} \biggl[ \chi_{[-1,1]} (x)\, dx + \sum_{n=1}^\infty e^{-n^2} \delta_{x_n}\biggr]
\end{equation}
where $\{x_n\}$ is a dense subset of $[-2,2]\setminus (-1,1)$. Then, as in Example~5.8 of \cite{EqMC},
$\rho_\infty$ exists and is the equilibrium measure for $[-1,1]$ (not $\fre=[-2,2]$). Moreover, the
method of \cite{Lub} shows that for $x\in (-1,1)$,
\begin{equation} \lb{7.5}
\f{1}{n+1}\, K_n (x,x) \to \f{\rho_\infty (x)}{N^{-1}}
\end{equation}
Using either the method of this paper (i.e., of \cite{Lub-jdam}) or the method of \cite{Lub}, one proves
universality with $\rho_\infty$.

\smallskip
3. \ Example~5.8 of \cite{EqMC} provides a measure with $\sigma_\ess(\mu)=[-2,2]$ but $\Sigma_\ac =
[-2,0]$ and where $\nu_n$ has multiple weak limits, including the equilibrium measures for $[-2,0]$
and for $[-2,2]$. By general principles \cite{StT}, the set of limits is connected, so uncountable.
One would like to prove that quasi-clock behavior nevertheless holds for the a.c.\ spectrum of
this model as this will
provide a key test for the conjecture that quasi-clock behavior always holds on $\Sigma_\ac$.

\smallskip
4. \ What has sometimes been called the Schr\"odinger conjecture (see \cite{MMG}) says that for any Jacobi
matrix  and a.e.\ $x\in\Sigma_\ac(\mu)$, we have a solution, $u_n$, with
\begin{equation} \lb{7.6}
0 < \inf_n \, \abs{u_n} \leq \sup_n \, \abs{u_n} <\infty
\end{equation}
and $u_{-1}=0$. Invariance of $\Sigma_\ac$ under rank one perturbations then proves that for a.e.\ $x\in
\Sigma_\ac(\mu)$, the transfer matrix is bounded. Thus, Theorem~3 in the strong form would always be
applicable.

\smallskip
5. \ While \eqref{6.16} is harmless since it only eliminates at most one $x$, one can ask if \eqref{6.17}
holds even if \eqref{6.16} fails. Using periodic problems, it is easy to construct ergodic cases where
$\arg u_n^+ =-\pi n/2$, so \eqref{6.29} provides no information on $I(\omega)$. Nevertheless, in these
cases, one can show $R(\omega)=I(\omega)=0$. We have not been able to find an example where for a set of
positive measure $\omega$'s, $s_{2n}(\omega)=n\pi$, $s_{2n+1}(\omega) = n\pi+\varphi$ with $\varphi$ some
fixed point in $(0,\pi)\setminus \{\f{\pi}{2}\}$. In that case, it might happen that $R(\omega)\neq 0$,
$I(\omega)\neq 0$. So it remains open if we need to exclude the $x$ with \eqref{6.16}.

\smallskip
6. \ While we could use soft methods in Section~\ref{s3}, at one point in our research we used an explicit
formula for the derivative of $\f{1}{n} K_n(x_0 + \f{a}{n}, x_0 + \f{a}{n})$ as a function of $a$ that may
be useful in other contexts, so we want to mention it. We start with a variation of parameters formula
(discussed, e.g., in \cite{JL,KKL}) that, in terms of the second kind polynomials of \eqref{1.38},
\begin{equation} \lb{7.7}
p_n(x)-p_n(x_0) = (x-x_0) \sum_{m=0}^{n-1} (p_n(x_0) q_m(x_0)-p_m(x_0) q_n(x_0)) p_m(x)
\end{equation}
which implies
\begin{equation} \lb{7.8}
p'_n (x_0) =\sum_{m=0}^{n-1} (p_n(x_0) q_m(x_0) - p_m(x_0) q_n(x_0)) p_m(x_0)
\end{equation}

Since
\begin{equation} \lb{7.9}
\left. \f{d}{da}\, \f{1}{n}\, K_n \biggl(x_0 + \f{a}{n}\, , x_0 + \f{a}{n}\biggr) \right|_{a=0} =
\f{1}{n^2} \sum_{j=0}^n 2p'_j(x_0) p_j(x_0)
\end{equation}
this leads to
\begin{equation} \lb{7.10}
\begin{split}
&\left. \f{d}{da}\,  \f{1}{n}\, K_n\biggl(x_0 + \f{a}{n}, x_0 + \f{a}{n}\biggr) \right|_{a=0} \\
&\quad = \f{2}{n^2} \sum_{j=0}^n \biggl[ p_j(x_0)^2 \biggl(\, \sum_{k=0}^j p_k(x_0) q_k(x_0)\biggr)
-q_j (x_0) p_j(x_0) \biggl(\, \sum_{k=0}^j p_k(x_0)^2 \biggr) \biggr]
\end{split}
\end{equation}

As noted in \cite{CD}, if $\f{1}{n} \sum_{j=0}^n p_j(x_0)^2$ and $\f{1}{n} \sum_{j=0}^n p_j(x_0) q_j(x_0)$
have limits and $\sup_n [\f{1}{n} \sum_{j=0}^n q_j(x_0)^2]<\infty$, then the right side of \eqref{7.10}
goes to $0$.

\bigskip


\begin{thebibliography}{999}

%
\bi{Av-prep1} A.~Avila, {\it Absolutely continuous spectrum for the almost
Mathieu operator}, preprint.

%
\bi{AD} A.~Avila and D.~Damanik,
{\it  Absolute continuity of the integrated density of states for
the almost Mathieu operator with non-critical coupling},
Invent.\ Math. {\bf 172} (2008), 439--453.

\bibitem{AFK} A.~Avila, B.~Fayad, and R.~Krikorian, {\it A KAM scheme for
$\mathrm{SL}(2,\R)$ cocycles with Liouvillean frequencies}, in preparation.

%
\bi{AJ} A.~Avila and S.~Jitomirskaya,
{\it Almost localization and almost reducibility},
to appear in J.\ Eur.\ Math.\ Soc.

%
\bi{AvKr} A.~Avila and R.~Krikorian,
{\it Reducibility or nonuniform hyperbolicity for quasiperiodic Schr\"odinger cocycles},
Ann.\ of Math. {\bf 164} (2006), 911--940.


%
\bi{Bour} J.~Bourgain,
{\it On the spectrum of lattice Schr\"odinger operators with deterministic potential},
J.\ Anal.\ Math. {\bf 87} (2002), 37--75.

%
\bi{Bour2002} J.~Bourgain,
{\it On the spectrum of lattice Schr\"odinger operators with deterministic potential. II.},
J.\ Anal.\ Math. {\bf 88} (2002), 221--254.

%
\bi{BJ} J.~Bourgain and S.~Jitomirskaya,
{\it Continuity of the Lyapunov exponent for quasiperiodic operators with
analytic potential},
J.\ Statist.\ Phys. {\bf 108} (2002),  1203--1218.

%
\bi{Dam} D.~Damanik,
{\it Lyapunov exponents and spectral analysis of ergodic Schr\"odinger operators: A survey
of Kotani theory and its applications},
in ``Spectral Theory and Mathematical Physics: A Festschrift in
Honor of Barry Simon's 60th birthday,'' pp.~539--563,
Proc.\ Sympos.\ Pure Math., {\bf 76.2}, American Mathematical Society, Providence, RI,
2007.

%
\bi{S169} P.~A.~Deift and B.~Simon,
{\it Almost periodic Schr\"odinger operators, III. The absolutely continuous
spectrum in one dimension},
Comm.\ Math.\ Phys. {\bf 90} (1983), 389--411.

%
\bi{Dom78} J.~Dombrowski,
{\it Quasitriangular matrices},
Proc.\ Amer.\ Math.\ Soc. {\bf 69} (1978), 95--96.

%
\bi{ET} P.~Erd\H{o}s and P.~Tur\'an,
{\it On interpolation. III.\ Interpolatory theory of polynomials},
Ann.\ of Math. (2) {\bf 41} (1940), 510--553.

%
\bi{FK2005} A.~Fedotov and F.~Klopp,
{\it Strong resonant tunneling, level repulsion and spectral type for
one-dimensional adiabatic quasi-periodic Schr\"odinger operators},
Ann.\ Sci.\ \'Ecole Norm.\ Sup. (4) {\bf 38} (2005), 889--950.

%
\bi{FK2006} A.~Fedotov and F.~Klopp,
{\it Weakly resonant tunneling interactions for adiabatic quasi-periodic
Schr\"odinger operators},
M\'em.\ Soc.\ Math.\ Fr. (N.S.), No.~104 (2006).

%
\bi{FrBk} G.~Freud,
\textit{Orthogonal Polynomials},
Pergamon Press, Oxford-New York, 1971.

%
\bi{Jit07} S.~Jitomirskaya,
{\it Ergodic Schr\"odinger Operators (on one foot)},
in ``Spectral Theory and Mathematical Physics: A Festschrift in Honor
of Barry Simon's 60th Birthday,''  pp.~613--647,
Proc.\ Symp.\ Pure Math., {\bf 76.2}, American Mathematical Society, Providence,
RI, 2007.

%
\bi{JL} S.~Jitomirskaya and Y.~Last,
{\it Power-law subordinacy and singular spectra, I.\ Half-line operators},
Acta Math. {\bf 183} (1999), 171--189.

%
\bi{JoMo} R.~Johnson and J.~Moser,
{\it The rotation number for almost periodic potentials},
Comm.\ Math.\ Phys. {\bf 84} (1982),  403--438.

%
\bi{KKL} R.~Killip, A.~Kiselev, and Y.~Last,
{\it Dynamical upper bounds on wavepacket spreading},
Amer.\ J.\ Math. {\bf 125} (2003),  1165--1198.

%
\bi{Kot97} S.~Kotani,
{\it Generalized Floquet theory for stationary Schr\"odinger operators in one dimension},
Chaos Solitons Fractals {\bf 8} (1997), 1817--1854.

%
\bi{KV} A.~B.~Kuijlaars and M.~Vanlessen,
{\it Universality for eigenvalue correlations from the modified Jacobi unitary ensemble},
Int.\ Math.\ Res.\ Not. {\bf 30} (2002), 1575--1600.

%
\bi{Fine4} Y.~Last and B.~Simon,
{\it Fine structure of the zeros of orthogonal polynomials,
IV.\ A priori bounds and clock behavior},
Commun.\ Pure Appl. Math. {\bf 61} (2008), 486--538.

%
\bi{LL} E.~Levin and D.~S.~Lubinsky,
{\it Applications of universality limits to zeros and reproducing kernels of
orthogonal polynomials},
J.\ Approx.\ Theory {\bf 150} (2008), 69--95.

%
\bi{Lub2008} D.~Lubinsky,
{\it A new approach to universality at the edge of the spectrum},
in ``Integrable Systems and Random Matrices: In honor of Percy Deift's 60th birthday,''
pp.~281--290, Contemporary Mathematics, {\bf 458},
American Mathematical Society, Providence, RI, 2008.

%
\bi{Lub} D.~S.~Lubinksy,
{\it A new approach to universality limits involving orthogonal polynomials},
to appear in Ann.\ of Math.

%
\bi{Lub-jdam} D.~S.~Lubinsky,
{\it Universality limits in the bulk for arbitrary measures on
compact sets},
to appear in J.\ Anal.\ Math.

%
\bi{LB} J.~Lund and K.~L.~Bowers,
\textit{Sinc Methods for Quadrature and Differential Equations},
Society for Industrial and Applied Mathematics (SIAM), Philadelphia, PA, 1992.

%
\bi{Markov} A.~A.~Markov,
{\it D\'emonstration de certaines in\'egalit\'es de M.~Tch\'ebychef},
Math.\ Ann. {\bf 24} (1884), 172--180.

%
\bi{MMG} V.~P.~Maslov, S.~A.~Molchanov, and A.~Ya.~Gordon,
{\it Behavior of generalized eigenfunctions at infinity and the Schr\"odinger
conjecture},
Russian J.\ Math.\ Phys. {\bf 1} (1993), 71--104.

%
\bi{MNT91} A.~M\'at\'e, P.~Nevai, and V.~Totik,
{\it Szeg\H{o}'s extremum problem on the unit circle},
Ann.\ of Math. {\bf 134} (1991), 433--453.

%
\bi{RS4} M.~Reed and B.~Simon,
\textit{Methods of Modern Mathematical Physics, IV.\ Analysis of
Operators},
Academic Press, New York, 1978.

%
\bi{Rudin} W.~Rudin,
\textit{Real and Complex Analysis}, 3rd edition, McGraw--Hill, New York, 1987.

%
\bi{S168} B.~Simon,
{\it Kotani theory for one dimensional stochastic Jacobi matrices},
Comm.\ Math.\ Phys. {\bf 89} (1983), 227--234.

%
\bi{Fine1} B.~Simon,
{\it Fine structure of the zeros of orthogonal polynomials, I.\ A tale of two pictures},
Electronic Transactions on Numerical Analysis {\bf 25} (2006),
328--368.

%
\bi{Fine2} B.~Simon,
{\it Fine structure of the zeros of orthogonal polynomials, II.\ OPUC with competing
exponential decay},
J.\ Approx.\ Theory {\bf 135} (2005), 125--139.

%
\bi{Fine3} B.~Simon,
{\it Fine structure of the zeros of orthogonal polynomials, III.\ Periodic recursion coefficients},
Commun.\ Pure Appl. Math. {\bf 59} (2006) 1042--1062.

%
\bi{EqMC} B.~Simon,
{\it Equilibrium measures and capacities in spectral theory},
Inverse Problems and Imaging {\bf 1} (2007), 713--772.

%
\bi{CD} B.~Simon,
{\it The Christoffel--Darboux kernel},
to appear in ``Perspectives in PDE, Harmonic Analysis and Applications''
in honor of V.~G.~Maz'ya's 70th birthday,
to be published in Proceedings of Symposia in Pure Mathematics.

%
\bi{2exts} B.~Simon,
{\it Two extensions of Lubinsky's universality theorem},
to appear in J.\ Anal.\ Math.

%
\bi{weak-CD} B.~Simon,
{\it Weak convergence of CD kernels and applications},
to appear in Duke Math.\ J.

%
\bi{Rice} B.~Simon,
\textit{Szeg\H{o}'s Theorem and Its Descendants: Spectral Theory for
$L^2$ Perturbations of Orthogonal Polynomials}, in preparation;
to be published by Princeton University Press.

%
\bi{StT} H.~Stahl and V.~Totik,
\textit{General Orthogonal Polynomials},
in ``Encyclopedia of Mathematics and its Applications," {\bf 43},
Cambridge University Press, Cambridge, 1992.

%
\bi{SzBk} G.~Szeg\H{o},
\textit{Orthogonal Polynomials}, Amer.\ Math.\ Soc.\ Colloq.\ Publ.
\textbf{23}, American Mathematical Society, Providence, R.I., 1939;
third ed., 1967.

%
\bi{Tit} E.~C.~Titchmarsh,
\textit{The Theory of Functions},
Oxford University Press, Oxford, 1932.

%
\bi{Tot} V.~Totik,
{\it Asymptotics for Christoffel functions for general measures on the real line},
J.\ Anal.\ Math. {\bf 81} (2000), 283--303.

%
\bi{Tot-acta} V.~Totik,
{\it Polynomial inverse images and polynomial inequalities},
Acta Math. {\bf 187} (2001), 139--160.

%
\bi{Tot-prep} V.~Totik,
{\it Universality and fine zero spacing on general sets},
in preparation.

%
\bi{vA} W.~Van Assche,
{\it Invariant zero behaviour for orthogonal polynomials on compact sets of the
real line},
Bull.\ Soc.\ Math.\ Belg.\ Ser.\ B {\bf 38} (1986), 1--13.

%
\bi{Wid} H.~Widom,
{\it Polynomials associated with measures in the complex plane},
J.\ Math.\ Mech. {\bf 16} (1967), 997--1013.


\end{thebibliography}
\end{document}